\documentclass[10pt]{article}

\usepackage{color}
\usepackage{amsmath}
\usepackage{graphicx}
\usepackage{footnote}
\usepackage{subscript}
\usepackage{amsfonts,color}
\usepackage{amsmath,amssymb}
\usepackage{epsfig}
\usepackage{amssymb} 
\usepackage{mathtools}
\usepackage{amsthm,wasysym}
\usepackage{framed}

\usepackage[english]{babel}
\usepackage[utf8]{inputenc}
\makeatletter
\usepackage{float}
\usepackage{wrapfig}
\restylefloat{figure}

\makeatletter
\let\@fnsymbol\@arabic
\makeatother

\usepackage{bbm}
\newcommand{\id}{{\boldsymbol{\mathbbm{1}}}}

\newcommand{\tr}{{\rm tr}}
\newcommand{\dev}{{\rm dev}}
\newcommand{\sym}{{\rm sym}}
\newcommand{\skw}{{\rm skew}}

\def\dd{\displaystyle}

\newtheorem{theorem}{Theorem}[section]

\newtheorem{remark}[theorem]{Remark}
\newtheorem{proposition}[theorem]{Proposition}

\def\barr{\begin{array}}
\usepackage{lscape}

	\def\earr{\end{array}}
\def\bec#1{\begin{equation}\label{#1}}
\def\becn{\begin{equation*}}
\def\endec{\end{equation}}
\def\endecn{\end{equation*}}
\def\dd{\displaystyle}
\def\bfm#1{\mbox{\boldmat}}
\begin{document}
	
	\title{A linear isotropic  Cosserat  shell model  including terms up to  $O(h^5)$. Existence and uniqueness }
\author{ 
	 Ionel-Dumitrel Ghiba\thanks{Corresponding author:  Ionel-Dumitrel Ghiba,  \ Department of Mathematics,  Alexandru Ioan Cuza University of Ia\c si,  Blvd.
	 	Carol I, no. 11, 700506 Ia\c si,
	 	Romania; and  Octav Mayer Institute of Mathematics of the
	 	Romanian Academy, Ia\c si Branch,  700505 Ia\c si, email:  dumitrel.ghiba@uaic.ro}  \qquad and \qquad  Mircea B\^irsan\thanks{Mircea B\^irsan, \ \  Lehrstuhl f\"{u}r Nichtlineare Analysis und Modellierung, Fakult\"{a}t f\"{u}r Mathematik,
 	Universit\"{a}t Duisburg-Essen, Thea-Leymann Str. 9, 45127 Essen, Germany; and Department of Mathematics,  Alexandru Ioan Cuza University of Ia\c si,  Blvd.
 Carol I, no. 11, 700506 Ia\c si,
Romania;  email: mircea.birsan@uni-due.de} \qquad and \qquad   Patrizio Neff\,\thanks{Patrizio Neff,  \ \ Head of Lehrstuhl f\"{u}r Nichtlineare Analysis und Modellierung, Fakult\"{a}t f\"{u}r
		Mathematik, Universit\"{a}t Duisburg-Essen,  Thea-Leymann Str. 9, 45127 Essen, Germany, email: patrizio.neff@uni-due.de} 
}
\date{}
\maketitle
\begin{abstract}

In this paper we derive the linear elastic Cosserat shell model incorporating effects up to order $O(h^5)$ in the  shell thickness $h$ as a particular case of the  recently introduced  geometrically nonlinear elastic Cosserat shell model. The existence and uniqueness of the solution is  proven in  suitable admissible sets. To this end, inequalities of Korn-type for shells are established which allow to show coercivity in the Lax-Milgram theorem. We are also showing an existence and uniqueness result for a truncated $O(h^3)$ model. Main issue is the suitable treatment of the curved reference configuration of the shell. Some connections to the classical Koiter membrane-bending model are highlighted.
  \medskip
  
\end{abstract}

\tableofcontents

\section{Introduction}\setcounter{equation}{0}
In this paper we consider  the linearised formulation of the geometrically nonlinear Cosserat shell model including terms up to order $O(h^5)$ in the shell-thickness $h$ proposed previously in \cite{GhibaNeffPartI}. The Cosserat approach to shell theory (also called micropolar shell theory) was initiated by the Cosserat brothers, who were the first to elaborate a rigorous study on directed media \cite{Cosserat08,Cosserat09}. The  idea is to model a shell-like body as a deformable two-dimensional surface endowed with directors assigned to every point. In this respect, there are several approaches in  literature. 
For instance, Green and Naghdi \cite{Naghdi72} have elaborated a shell theory using two-dimensional surfaces endowed with a single deformable director also called \textit{Cosserat surfaces}. The theory of Cosserat surfaces has been presented in the monograph \cite{Rubin00} and the linearised theory has been investigated in a number of papers \cite{Davini75,Birsan-IJES-07,Birsan08,Birsan-EJM/S-09}.

Another direct approach to shell theory, also called the theory of \textit{simple shells} (or \textit{directed surfaces}), describes the shell-like body as a deformable surface endowed with an independent triad of orthonormal vectors connected to each point of the surface. The triad of directors characterizes the \textit{orientation} of material points and introduces thus the \textit{microrotation tensor}. 
The theory of simple shells has been presented by Zhilin and Altenbach in \cite{Zhilin76,Altenbach-Zhilin-88,Altenbach04,Zhilin06} and a mathematical study of the linearised equations for this model is included  in the papers \cite{Birsan-Alten-MMAS-10,Birsan-Alten-ZAMM-11}.

In  \cite{GhibaNeffPartI} we have established a novel geometrically nonlinear Cosserat shell model including terms up to order $O(h^5)$ in the shell-thickness $h$. The dimensional descent was obtained starting with a 3D-parent Cosserat model and assuming an appropriate 8-parameter ansatz for the shell-deformation through the thickness. This is the derivation approach and it has allowed us to arrive at specific novel strain and curvature measures.    In this way, we obtained a kinematical model which is equivalent to the kinematical model of 6-parameter shells. Nevertheless, the theory of 6-parameter shells was developed for shell-like bodies made of Cauchy materials, see the monographs \cite{Libai98,Pietraszkiewicz-book04} or the papers \cite{Eremeyev06,Pietraszkiewicz11}, but our model is expressed in terms of the accepted measures for the bending and the change of curvature (cf. the assertion of Acharya \cite{acharya2000nonlinear} and   Anicic and Leg\'{e}r \cite{anicic1999formulation}, respectively, see also  \cite{vsilhavycurvature}). For recent linear Naghdi shell models \cite{Naghdi69} which do not incorporate dedicated Cosserat effects we refer to \cite{Tambaca-14,Tambaca-16}, while for a two-dimensional model of elastic shell-like body derived from the three-dimensional linearized micropolar elasticity by using the asymptotic expansion method we refer to  \cite{Aganovic07}.

For all our proposed new linearised models we show  existence and uniqueness. Existence results for the linearised equations of 6-parameter shells have been proved in \cite{EremeyevLebedev11}.
We refer to the review paper \cite{Altenbach-Erem-Review} for a detailed presentation of various approaches and developments concerning Cosserat-type shell theories and to the books \cite{Ciarlet98b,Ciarlet00,Ciarlet2Diff-Geo2005} for shell theories in the classical linear elasticity framework.

In the linearised version of our Cosserat shell model we obtain the same linear strain measures as in the theory of 6-parameter shells, due to the fact that the kinematical structure is equivalent, since there is  an explicit dependence of the internal energy density on the change of curvature measure and on the bending measure.

 \section{The   geometrically nonlinear  unconstrained  Cosserat shell model including terms up to $O(h^5)$}\setcounter{equation}{0}\label{Intro}

\subsection{Notations}
 
   In this paper, 
 for $a,b\in\mathbb{R}^n$ we let $\bigl\langle {a},{b} \bigr\rangle _{\mathbb{R}^n}$  denote the scalar product on $\mathbb{R}^n$ with
 associated vector norm $\lVert a\rVert _{\mathbb{R}^n}^2=\bigl\langle {a},{a} \bigr\rangle _{\mathbb{R}^n}$. 
 The standard Euclidean scalar product on  the set of real $n\times  {m}$ second order tensors $\mathbb{R}^{n\times  {m}}$ is given by
 $\bigl\langle  {X},{Y} \bigr\rangle _{\mathbb{R}^{n\times  {m}}}={\rm tr}(X\, Y^T)$, and thus the  {(squared)} Frobenius tensor norm is
 $\lVert {X}\rVert ^2_{\mathbb{R}^{n\times  {m}}}=\bigl\langle  {X},{X} \bigr\rangle _{\mathbb{R}^{n\times  {m}}}$. The identity tensor on $\mathbb{R}^{n \times n}$ will be denoted by $\id_n$, so that
 ${\rm tr}({X})=\bigl\langle {X},{\id}_n \bigr\rangle $, and the zero matrix is denoted by $0_n$. We let ${\rm Sym}(n)$ and ${\rm Sym}^+(n)$ denote the symmetric and positive definite symmetric tensors, respectively.  We adopt the usual abbreviations of Lie-group theory, i.e.,
 ${\rm GL}(n)=\{X\in\mathbb{R}^{n\times n}\;|\det({X})\neq 0\}$ the general linear group ${\rm SO}(n)=\{X\in {\rm GL}(n)| X^TX=\id_n,\det({X})=1\}$ with
 corresponding Lie-algebras $\mathfrak{so}(n)=\{X\in\mathbb{R}^{n\times n}\;|X^T=-X\}$ of skew symmetric tensors
 and $\mathfrak{sl}(n)=\{X\in\mathbb{R}^{n\times n}\;| \,\tr({X})=0\}$ of traceless tensors. For all $X\in\mathbb{R}^{n\times n}$ we set ${\rm sym}\, X\,=\frac{1}{2}(X^T+X)\in{\rm Sym}(n)$, $\skw\,X\,=\frac{1}{2}(X-X^T)\in \mathfrak{so}(n)$ and the deviatoric part $\dev \,X\,=X-\frac{1}{n}\;\,\tr(X)\cdot\id_n\in \mathfrak{sl}(n)$  and we have
 the orthogonal Cartan-decomposition  of the Lie-algebra
 $
 \mathfrak{gl}(n)=\{\mathfrak{sl}(n)\cap {\rm Sym}(n)\}\oplus\mathfrak{so}(n) \oplus\mathbb{R}\!\cdot\! \id_n,$ $
 X=\dev\, \sym \,X\,+ \skw\,X\,+\frac{1}{n}\,\tr(X)\!\cdot\! \id_n\,.
 $ For vectors $\xi,\eta\in\mathbb{R}^n$, we have the tensor product
 $(\xi\otimes\eta)_{ij}=\xi_i\,\eta_j$. A matrix having the  three  column vectors $A_1,A_2, A_3$ will be written as 
 $
 (A_1\,|\, A_2\,|\,A_3).
 $ 
 For a given matrix $M\in \mathbb{R}^{2\times 2}$ we define the lifted quantities $
 M^\flat =\begin{footnotesize}\begin{pmatrix}
 M_{11}& M_{12}&0 \\
 M_{21}&M_{22}&0 \\
 0&0&0
 \end{pmatrix}\end{footnotesize}
 \in \mathbb{R}^{3\times 3}$ and $
 \widehat{M} =\begin{footnotesize}\begin{pmatrix}
 M_{11}& M_{12}&0 \\
 M_{21}&M_{22}&0 \\
 0&0&1
 \end{pmatrix}\end{footnotesize}
 \in \mathbb{R}^{3\times 3}$. We make use of the operator $\mathrm{axl}: \mathfrak{so}(3)\to\mathbb{R}^3$ associating with a matrix $A\in \mathfrak{so}(3)$ the vector $\mathrm{axl}({A}):=(-A_{23},A_{13},-A_{12})^T$. The inverse operator will be denoted by ${\rm Anti}: \mathbb{R}^3\to \mathfrak{so}(3)$.
 
 For  an open domain  $\Omega\subseteq\mathbb{R}^{3}$,
 the usual Lebesgue spaces of square integrable functions, vector or tensor fields on $\Omega$ with values in $\mathbb{R}$, $\mathbb{R}^3$, $\mathbb{R}^{3\times 3}$ or ${\rm SO}(3)$, respectively will be denoted by ${\rm L}^2(\Omega;\mathbb{R})$, ${\rm L}^2(\Omega;\mathbb{R}^3)$, ${\rm L}^2(\Omega; \mathbb{R}^{3\times 3})$ and ${\rm L}^2(\Omega; {\rm SO}(3))$, respectively. Moreover, we use the standard Sobolev spaces ${\rm H}^{1}(\Omega; \mathbb{R})$ \cite{Adams75,Raviart79,Leis86}
 of functions $u$.  For vector fields $u=\left(    u_1, u_2, u_3\right)^T$ with  $u_i\in {\rm H}^{1}(\Omega)$, $i=1,2,3$,
 we define
 $
 \nabla \,u:=\left(
 \nabla\,  u_1\,|\,
 \nabla\, u_2\,|\,
 \nabla\, u_3
 \right)^T.
 $
 The corresponding Sobolev-space will be denoted by
 $
 {\rm H}^1(\Omega; \mathbb{R}^{3})$. A tensor $Q:\Omega\to {\rm SO}(3)$ having the components in ${\rm H}^1(\Omega; \mathbb{R})$ belongs to ${\rm H}^1(\Omega; {\rm SO}(3))$. In writing the norm in the corresponding  Sobolev-space we will specify the space. The space will be omitted only when the Frobenius norm or scalar product is considered. 
 
 \subsection{Shell-kinematics}

 Let $\Omega_\xi\subset\mathbb{R}^3$ be a three-dimensional curved {\it shell-like thin domain}. Here, the domains $\Omega_h $ and $ \Omega_\xi $ are referred to a  fixed right Cartesian coordinate frame with unit vectors $
 \boldsymbol e_i$ along the axes $Ox_i$. A generic point of $\Omega_\xi$ will be denoted by $(\xi_1,\xi_2,\xi_3)$. The elastic material constituting the shell is assumed to be homogeneous and isotropic and the reference configuration $\Omega_\xi$ is assumed to be a natural state. 
 The deformation of the body occupying the domain $\Omega_\xi$ is described by a vector map $\varphi_\xi:\Omega_\xi\subset\mathbb{R}^3\rightarrow\mathbb{R}^3$ (\textit{called deformation}) and by a \textit{microrotation}  tensor
 $
 \overline{R}_\xi:\Omega_\xi\subset\mathbb{R}^3\rightarrow {\rm SO}(3)\, 
 $ attached at each point. 
 We denote the current configuration (deformed configuration) by $\Omega_c:=\varphi_\xi(\Omega_\xi)\subset\mathbb{R}^3$. We consider  the \textit{fictitious Cartesian (planar) configuration} of the body $\Omega_h $. This parameter domain $\Omega_h\subset\mathbb{R}^3$ is a right cylinder of the form
 $$\Omega_h=\left\{ (x_1,x_2,x_3) \,\Big|\,\, (x_1,x_2)\in\omega, \,\,\,-\dfrac{h}{2}\,< x_3<\, \dfrac{h}{2}\, \right\} =\,\,\dd\,\omega\,\times\left(-\frac{h}{2},\,\frac{h}{2}\right),$$
 where  $\omega\subset\mathbb{R}^2$ is a bounded domain with Lipschitz boundary
 $\partial \omega$ and the constant length $h>0$ is the \textit{thickness of the shell}.
 For shell--like bodies we consider   the  domain $\Omega_h $ to be {thin}, i.e. the thickness $h$ is {small}. 
 
 The diffeomorphism $\Theta:\mathbb{R}^3\rightarrow\mathbb{R}^3$ describing the reference configuration (i.e., the curved surface of the shell),  will be chosen in the specific form
 \begin{equation}\label{defTheta}
 \Theta(x_1,x_2,x_3)\,=\,y_0(x_1,x_2)+x_3\ n_0(x_1,x_2), \ \ \ \ \ \qquad  n_0\,=\,\dd\frac{\partial_{x_1}y_0\times \partial_{x_2}y_0}{\lVert \partial_{x_1}y_0\times \partial_{x_2}y_0\rVert}\, ,
 \end{equation}
 where $y_0:\omega\to \mathbb{R}^3$ is a function of class $C^2(\omega)$. If not otherwise indicated, by $\nabla\Theta$ we denote
 $\nabla\Theta(x_1,x_2,0)$.

 Now, let us  define the map
 $
 \varphi:\Omega_h\rightarrow \Omega_c,\  \varphi(x_1,x_2,x_3)=\varphi_\xi( \Theta(x_1,x_2,x_3)).
 $
 We view $\varphi$ as a function which maps the fictitious  planar reference configuration $\Omega_h$ into the deformed configuration $\Omega_c$.
 We also consider the \textit{elastic microrotation}
 $
 \overline{Q}_{e,s}:\Omega_h\rightarrow{\rm SO}(3),\  \overline{Q}_{e,s}(x_1,x_2,x_3):=\overline{R}_\xi(\Theta(x_1,x_2,x_3))\,.
 $

 The dimensional descent in \cite{GhibaNeffPartI} is done  by   assuming that  the elastic microrotation is constant through the thickness, i.e.
 $
 \overline{Q}_{e,s}(x_1,x_2,x_3)=\overline{Q}_{e,s}(x_1,x_2), 
 $
 and  by considering an \textit{8-parameter quadratic ansatz} in the thickness direction for the reconstructed total deformation $\varphi_s:\Omega_h\subset \mathbb{R}^3\rightarrow \mathbb{R}^3$ of the shell-like body, i.e.,
 \begin{align}\label{ansatz}
 \varphi_s(x_1,x_2,x_3)\,=\,&m(x_1,x_2)+\bigg(x_3\varrho_m(x_1,x_2)+\dd\frac{x_3^2}{2}\varrho_b(x_1,x_2)\bigg)\overline{Q}_{e,s}(x_1,x_2)\nabla\Theta.e_3\, .
 \end{align}
Here $m:\omega\subset\mathbb{R}^2\to\mathbb{R}^3$ represents the
 deformation of the total midsurface,  $\varrho_m,\,\varrho_b:\omega\subset\mathbb{R}^2\to \mathbb{R}$ allow in principal for symmetric thickness stretch  ($\varrho_m\neq1$) and asymmetric thickness stretch ($\varrho_b\neq 0$) about the midsurface and which are given by
 \begin{align}\label{final_rho}
 \varrho_m\,=\,&1-\frac{\lambda}{\lambda+2\,\mu}[\bigl\langle  \overline{Q}_{e,s}^T(\nabla m|0)[\nabla\Theta \,]^{-1},\id_3 \bigr\rangle -2]\;, \\
 \dd\varrho_b\,=\,&-\frac{\lambda}{\lambda+2\,\mu}\bigl\langle  \overline{Q}_{e,s}^T(\nabla (\,\overline{Q}_{e,s}\nabla\Theta \,.e_3)|0)[\nabla\Theta \,]^{-1},\id_3 \bigr\rangle   +\frac{\lambda}{\lambda+2\,\mu}\bigl\langle  \overline{Q}_{e,s}^T(\nabla m|0)[\nabla\Theta \,]^{-1}(\nabla n_0|0)[\nabla\Theta \,]^{-1},\id_3 \bigr\rangle .\notag
 \end{align}
 This allowed us to  obtained a fully two-dimensional minimization problem in which the reduced energy density is expressed in terms of the  following tensor fields (the same strain measures are also considered in \cite{Libai98,Pietraszkiewicz-book04,Eremeyev06,NeffBirsan13,Birsan-Neff-L54-2014} but with other motivations of their significance) on the surface $\omega\,$  
 \begin{align}\label{e55}
 \mathcal{E}_{m,s} & :\,=\,    \overline{Q}_{e,s}^T  (\nabla  m|\overline{Q}_{e,s}\nabla\Theta \,.e_3)[\nabla\Theta \,]^{-1}-\id_3\not\in {\rm Sym}(3),\qquad \quad \qquad\ \ \ \ \,  \text{{\it elastic shell strain tensor}} ,  \\
 \mathcal{K}_{e,s} & :\,=\,  \Big(\mathrm{axl}(\overline{Q}_{e,s}^T\,\partial_{x_1} \overline{Q}_{e,s})\,|\, \mathrm{axl}(\overline{Q}_{e,s}^T\,\partial_{x_2} \overline{Q}_{e,s})\,|0\Big)[\nabla\Theta \,]^{-1}\not\in {\rm Sym}(3) \quad \text{\it\  elastic shell bending--curvature tensor}.\notag
 \end{align}

\subsection{Geometrically nonlinear energy functional}

 In   \cite{GhibaNeffPartI} we have obtained the following two-dimensional minimization problem   for the
 deformation of the midsurface $m:\omega
 \,{\to}\,
 \mathbb{R}^3$ and the microrotation of the shell
 $\overline{Q}_{e,s}:\omega
 \,{\to}\,
 \textrm{SO}(3)$ solving on $\omega
 \,\subset\mathbb{R}^2
 $: minimize with respect to $ (m,\overline{Q}_{e,s}) $ the  functional
 \begin{equation}\label{e89}
 I(m,\overline{Q}_{e,s})\!=\!\! \int_{\omega}   \!\!\Big[  \,
 W_{\mathrm{memb}}\big(  \mathcal{E}_{m,s}  \big) +W_{\mathrm{memb,bend}}\big(  \mathcal{E}_{m,s} ,\,  \mathcal{K}_{e,s} \big)   +
 W_{\mathrm{bend,curv}}\big(  \mathcal{K}_{e,s}    \big)
 \Big] \,\underbrace{{\rm det}(\nabla y_0|n_0)}_{{\rm det}\nabla\Theta \,}       \, d a - \overline{\Pi}(m,\overline{Q}_{e,s})\,,
 \end{equation}
 where the  membrane part $\,W_{\mathrm{memb}}\big(  \mathcal{E}_{m,s} \big) \,$, the membrane--bending part $\,W_{\mathrm{memb,bend}}\big(  \mathcal{E}_{m,s} ,\,  \mathcal{K}_{e,s} \big) \,$ and the bending--curvature part $\,W_{\mathrm{bend,curv}}\big(  \mathcal{K}_{e,s}    \big)\,$ of the shell energy density are given by
 \begin{align}\label{e90}
 W_{\mathrm{memb}}\big(  \mathcal{E}_{m,s} \big)=& \, \Big(h+{\rm K}\,\dfrac{h^3}{12}\Big)\,
 W_{\mathrm{shell}}\big(    \mathcal{E}_{m,s} \big),\vspace{2.5mm}\notag\\    
 W_{\mathrm{memb,bend}}\big(  \mathcal{E}_{m,s} ,\,  \mathcal{K}_{e,s} \big)=& \,   \Big(\dfrac{h^3}{12}\,-{\rm K}\,\dfrac{h^5}{80}\Big)\,
 W_{\mathrm{shell}}  \big(   \mathcal{E}_{m,s} \, {\rm B}_{y_0} +   {\rm C}_{y_0} \mathcal{K}_{e,s} \big)  \\&
 -\dfrac{h^3}{3} \mathrm{ H}\,\mathcal{W}_{\mathrm{shell}}  \big(  \mathcal{E}_{m,s} ,
 \mathcal{E}_{m,s}{\rm B}_{y_0}+{\rm C}_{y_0}\, \mathcal{K}_{e,s} \big)+
 \dfrac{h^3}{6}\, \mathcal{W}_{\mathrm{shell}}  \big(  \mathcal{E}_{m,s} ,
 ( \mathcal{E}_{m,s}{\rm B}_{y_0}+{\rm C}_{y_0}\, \mathcal{K}_{e,s}){\rm B}_{y_0} \big)\vspace{2.5mm}\notag\\&+ \,\dfrac{h^5}{80}\,\,
 W_{\mathrm{mp}} \big((  \mathcal{E}_{m,s} \, {\rm B}_{y_0} +  {\rm C}_{y_0} \mathcal{K}_{e,s} )   {\rm B}_{y_0} \,\big),  \vspace{2.5mm}\notag\\
 W_{\mathrm{bend,curv}}\big(  \mathcal{K}_{e,s}    \big) = &\,  \,\Big(h-{\rm K}\,\dfrac{h^3}{12}\Big)\,
 W_{\mathrm{curv}}\big(  \mathcal{K}_{e,s} \big)    +  \Big(\dfrac{h^3}{12}\,-{\rm K}\,\dfrac{h^5}{80}\Big)\,
 W_{\mathrm{curv}}\big(  \mathcal{K}_{e,s}   {\rm B}_{y_0} \,  \big)  + \,\dfrac{h^5}{80}\,\,
 W_{\mathrm{curv}}\big(  \mathcal{K}_{e,s}   {\rm B}_{y_0}^2  \big),\notag
 \end{align}
 and
 \begin{align}\label{quadraticforms}
 W_{\mathrm{shell}}( X) & =   \mu\,\lVert  \mathrm{sym}\,X\rVert ^2 +  \mu_{\rm c}\lVert \mathrm{skew}\,X\rVert ^2 +\dfrac{\lambda\,\mu}{\lambda+2\,\mu}\,\big[ \mathrm{tr}   (X)\big]^2\\
 &=  \mu\, \lVert  \mathrm{  dev \,sym} \,X\rVert ^2  +  \mu_{\rm c} \lVert  \mathrm{skew}   \,X\rVert ^2 +\,\dfrac{2\,\mu\,(2\,\lambda+\mu)}{3(\lambda+2\,\mu)}\,[\mathrm{tr}  (X)]^2\\
 \mathcal{W}_{\mathrm{shell}}(  X,  Y)& =   \mu\,\bigl\langle  \mathrm{sym}\,X,\,\mathrm{sym}\,   \,Y \bigr\rangle   +  \mu_{\rm c}\bigl\langle \mathrm{skew}\,X,\,\mathrm{skew}\,   \,Y \bigr\rangle   +\,\dfrac{\lambda\,\mu}{\lambda+2\,\mu}\,\mathrm{tr}   (X)\,\mathrm{tr}   (Y),  \notag\vspace{2.5mm}\\
 W_{\mathrm{mp}}(  X)&= \mu\,\lVert \mathrm{sym}\,X\rVert ^2+  \mu_{\rm c}\lVert \mathrm{skew}\,X\rVert ^2 +\,\dfrac{\lambda}{2}\,\big[  \tr(X)\,\big]^2=
 \mathcal{W}_{\mathrm{shell}}(  X)+ \,\dfrac{\lambda^2}{2\,(\lambda+2\,\mu)}\,[\mathrm{tr} (X)]^2,\notag\vspace{2.5mm}\\
 W_{\mathrm{curv}}(  X )&=\mu\, L_{\rm c}^2 \left( b_1\,\lVert  \dev\,\text{sym} \,X\rVert ^2+b_2\,\lVert \text{skew}\,X\rVert ^2+b_3\,
 [\tr (X)]^2\right), \quad \forall\, X,Y\in \mathbb{R}^{3\times 3}.\notag
 \end{align}
  In the formulation of the minimization problem we  have considered the  {\it Weingarten map (or shape operator)}  defined by 
 $
 {\rm L}_{y_0}\,=\, {\rm I}_{y_0}^{-1} {\rm II}_{y_0}\in \mathbb{R}^{2\times 2},
 $
 where ${\rm I}_{y_0}:=[{\nabla  y_0}]^T\,{\nabla  y_0}\in \mathbb{R}^{2\times 2}$ and  ${\rm II}_{y_0}:\,=\,-[{\nabla  y_0}]^T\,{\nabla  n_0}\in \mathbb{R}^{2\times 2}$ are  the matrix representations of the {\it first fundamental form (metric)} and the  {\it  second fundamental form} of the surface, respectively.  
 Then, the {\it Gau{\ss} curvature} ${\rm K}$ of the surface  is determined by
 $
 {\rm K} :=\,{\rm det}\,{{\rm L}_{y_0}}\, 
 $
 and the {\it mean curvature} $\,{\rm H}\,$ through
 $
 2\,{\rm H}\, :={\rm tr}({{\rm L}_{y_0}}).
 $  We have also used  the  tensors defined by
 \begin{align}\label{AB}
 {\rm A}_{y_0}&:=(\nabla y_0|0)\,\,[\nabla\Theta \,]^{-1}\in\mathbb{R}^{3\times 3}, \qquad \qquad 
 {\rm B}_{y_0}:=-(\nabla n_0|0)\,\,[\nabla\Theta \,]^{-1}\in\mathbb{R}^{3\times 3},
 \end{align}
 and the so-called \textit{{alternator tensor}} ${\rm C}_{y_0}$ of the surface \cite{Zhilin06}
 \begin{align}
 {\rm C}_{y_0}:=\det\nabla\Theta \,\, [	\nabla\Theta \,]^{-T}\,\begin{footnotesize}\begin{pmatrix}
 0&1&0 \\
 -1&0&0 \\
 0&0&0
 \end{pmatrix}\end{footnotesize}\,  [	\nabla\Theta \,]^{-1}.
 \end{align}
 
 The parameters $\mu$ and $\lambda$ are the \textit{Lam\'e constants}
 of classical isotropic elasticity, $\kappa=\frac{2\,\mu+3\,\lambda}{3}$ is the \textit{infinitesimal bulk modulus}, $b_1, b_2, b_3$ are \textit{non-dimensional constitutive curvature coefficients (weights)}, $\mu_{\rm c}\geq 0$ is called the \textit{{Cosserat couple modulus}} and ${L}_{\rm c}>0$ introduces an \textit{{internal length} } which is {characteristic} for the material, e.g., related to the grain size in a polycrystal. The
 internal length ${L}_{\rm c}>0$ is responsible for \textit{size effects} in the sense that smaller samples are relatively stiffer than
 larger samples. If not stated otherwise, we assume that $\mu>0$, $\kappa>0$, $\mu_{\rm c}>0$, $b_1>0$, $b_2>0$, $b_3> 0$. All the constitutive coefficients  are coming from the three-dimensional Cosserat formulation, without using any a posteriori fitting of some two-dimensional constitutive coefficients.

 The potential of applied external loads $ \overline{\Pi}(m,\overline{Q}_{e,s}) $ appearing in \eqref{e89} is expressed by 
 \begin{align}\label{e4o}
 \overline{\Pi}(m,\overline{Q}_{e,s})\,=\,& \, \Pi_\omega(m,\overline{Q}_{e,s}) + \Pi_{\gamma_t}(m,\overline{Q}_{e,s})\,,\qquad\textrm{with}   \\
 \Pi_\omega(m,\overline{Q}_{e,s}) \,=\,& \dd\int_{\omega}\bigl\langle  {f}, u \bigr\rangle \, da + \Lambda_\omega(\overline{Q}_{e,s})\qquad \text{and}\qquad
 \Pi_{\gamma_t}(m,\overline{Q}_{e,s})\,=\, \dd\int_{\gamma_t}\bigl\langle  {t},  u \bigr\rangle \, ds + \Lambda_{\gamma_t}(\overline{Q}_{e,s})\,,\notag
 \end{align}
 where $ u(x_1,x_2) \,=\, m(x_1,x_2)-y_0(x_1,x_2) $ is the displacement vector of the midsurface,  $\Pi_\omega(m,\overline{Q}_{e,s})$ is the potential of the external surface loads $f$, while  $\Pi_{\gamma_t}(m,\overline{Q}_{e,s})$ is the potential of the external boundary loads $t$.  The functions $\Lambda_\omega\,, \Lambda_{\gamma_t} : {\rm L}^2 (\omega, \textrm{SO}(3))\rightarrow\mathbb{R} $ are expressed in terms of the loads from the three-dimensional parental variational problem, see \cite{GhibaNeffPartI}, and they are assumed to be continuous and bounded operators. Here, $ \gamma_t $ and $ \gamma_d $ are nonempty subsets of the boundary of $ \omega $ such that $   \gamma_t \cup \gamma_d= \partial\omega $ and $ \gamma_t \cap \gamma_d= \emptyset $\,. On $ \gamma_t $ we have considered traction boundary conditions, while on $ \gamma_d $ we have the Dirichlet-type boundary conditions: \begin{align}\label{boundary}
 m\big|_{\gamma_d}&=m^*,\ \  \ \ \ \ \ \ \text{simply supported (fixed, welded)}, \qquad \qquad 
 \overline{Q}_{e,s}\big|_{\gamma_d}=\overline{Q}_{e,s}^*,\ \  \ \ \ \ \ \ \text{clamped},
 \end{align}
 where the boundary conditions are to be understood in the sense of traces.

    In our model the total energy is not simply the sum of  energies coupling the membrane and the bending effect, respectively.  Two further coupling energies are still present and they result directly from the dimensional reduction of the variational problem from the geometrically nonlinear three-dimensional Cosserat elasticity.  
 	Our model   is constructed in \cite{GhibaNeffPartI} under the following assumptions upon the thickness 
 	\begin{align}\label{ch5in}
 	h\,\max \{\sup_{(x_1,x_2)\in {\omega}}|{\kappa_1}|,\sup_{(x_1,x_2)\in {\omega}}|{\kappa_2}|\}<2\end{align}
 	where  ${\kappa_1}$ and ${\kappa_2}$ denote  the {\rm principal curvatures} of the surface.

The model admits global minimizers for   materials with positive Cosserat couple modulus   $\,\mu_{\rm c}>0$ and the Poisson ratio  $\nu=\frac{\lambda}{2(\lambda+\mu)}$ and Young's modulus $ E=\frac{\mu(3\,\lambda+2\,\mu)}{\lambda+\mu}$ are such that\footnote{These conditions are equivalent to $\mu>0$ and $2\,\lambda+\mu> 0$.}
$
-\frac{1}{2}<\nu<\frac{1}{2}$ \text{and} $E>0\, 
$  \cite{GhibaNeffPartII}.  
Under these assumptions on the constitutive coefficients, together with the positivity of $\mu$, $\mu_{\rm c}$, $b_1$, $b_2$ and $b_3$, and the orthogonal Cartan-decomposition  of the Lie-algebra
$
\mathfrak{gl}(3)$ and with the definition
\begin{align}\label{e78}
{W}_{\mathrm{shell}}( X)
:= &\, {W}_{\mathrm{shell}}^{\infty}( \sym \,X) +  \mu_{\rm c} \lVert  \mathrm{skew}   \,X\rVert ^2 \quad \ \forall \, X\in\mathbb{R}^{3\times 3},\\  {W}_{\mathrm{shell}}^{\infty}( S):= &\,  \mu\, \lVert  S\rVert ^2   +\,\dfrac{\lambda\,\mu}{\lambda+2\,\mu}\,\big[ \mathrm{tr}   (S)\big]^2\qquad \quad \, \forall \, S\,\in{\rm Sym}(3),\notag
\end{align}
it follows that there exists positive constants  $c_1^+, c_2^+, C_1^+$ and $C_2^+$  such that for all $X\in \mathbb{R}^{3\times 3}$ the following inequalities hold
\begin{align}\label{pozitivdef}
C_1^+ \lVert S\rVert ^2&\geq\, {W}_{\mathrm{shell}}^{\infty}( S) \geq\, c_1^+ \lVert  S\rVert ^2 \qquad \qquad \qquad \qquad \qquad \qquad\ \forall \, S\,\in{\rm Sym}(3),\notag\\
C_1^+ \lVert \sym\,X\rVert ^2+\mu_{\rm c}\,\lVert \skw\,X\rVert ^2&\geq\, W_{\mathrm{shell}}(  X) \geq\, c_1^+ \lVert  \sym\,X\rVert ^2+\mu_{\rm c}\,\lVert \skw\,X\rVert ^2 \quad \qquad\forall \, X\in\mathbb{R}^{3\times 3},\\
C_2^+ \lVert X \rVert ^2 &\geq\, W_{\mathrm{curv}}(  X )
\geq\,  c_2^+ \lVert X \rVert ^2\qquad \qquad \qquad \qquad \qquad \qquad  \forall \, X\in\mathbb{R}^{3\times 3}.\notag
\end{align}
Here,  $c_1^+$ and $C_1^+$ denote the  smallest and the largest eigenvalues, respectively, of the quadratic form ${W}_{\mathrm{shell}}^{\infty}( X)$. Hence, they are independent of $\mu_{\rm c}$.
\subsection{Preliminary results}
In  the proof of the existence result for geometrically nonlinear model, the condition on the thickness $h$ is used  only in the step where the   coercivity of the internal energy density is deduced, see \cite{GhibaNeffPartII,GhibaNeffPartIII}. We recall the results since they will be useful to establish corresponding existence results in the linearised models, too.
\begin{proposition}\label{propcoerh5} {\rm [Coercivity in the theory including terms up to order $O(h^5)$]} For sufficiently small values of the thickness $h$ such that  
	\begin{align}\label{rcondh5}
	h\max\{\sup_{x\in\omega}|\kappa_1|, \sup_{x\in\omega}|\kappa_2|\}<\alpha \qquad \text{with}\qquad  \alpha<\sqrt{\frac{2}{3}(29-\sqrt{761})}\simeq 0.97083
	\end{align} 
	and for constitutive coefficients  satisfying  $\mu>0, \,\mu_{\rm c}>0$, $2\,\lambda+\mu> 0$, $b_1>0$, $b_2>0$ and $b_3>0$,   the  energy density
	\begin{align} W(\mathcal{E}_{m,s}, \mathcal{K}_{e,s})=W_{\mathrm{memb}}\big(  \mathcal{E}_{m,s} \big)+W_{\mathrm{memb,bend}}\big(  \mathcal{E}_{m,s} ,\,  \mathcal{K}_{e,s} \big)+W_{\mathrm{bend,curv}}\big(  \mathcal{K}_{e,s}    \big) 
	\end{align}
	is coercive in the sense that  there exists a constant   $a_1^+>0$  such that
	$	W(\mathcal{E}_{m,s}, \mathcal{K}_{e,s})\,\geq\, a_1^+\, \big( \lVert \mathcal{E}_{m,s}\rVert ^2 + \lVert \mathcal{K}_{e,s}\rVert ^2\,\big)$,
	where
	$a_1^+$ depends on the constitutive coefficients. The following inequality holds true
	\begin{align}
			W(\mathcal{E}_{m,s}, \mathcal{K}_{e,s})\,\geq\, a_1^+\, \big( \lVert \mathcal{E}_{m,s}\rVert ^2 +\lVert
			\mathcal{E}_{m,s}{\rm B}_{y_0}+{\rm C}_{y_0}\, \mathcal{K}_{e,s}  \rVert ^2+ \lVert \mathcal{K}_{e,s}\rVert ^2\,\big).
	\end{align}
	\hfill$\blacksquare$
\end{proposition}
In the geometrically nonlinear Cosserat shell model up to $O(h^3)$
the shell energy density $W^{(h^3)}(\mathcal{E}_{m,s}, \mathcal{K}_{e,s})$ is given by 
\begin{align}\label{h3energy} W^{(h^3)}(\mathcal{E}_{m,s}, \mathcal{K}_{e,s})=&\,  \Big(h+{\rm K}\,\dfrac{h^3}{12}\Big)\,
W_{\mathrm{shell}}\big(    \mathcal{E}_{m,s} \big)+  \dfrac{h^3}{12}\,
W_{\mathrm{shell}}  \big(   \mathcal{E}_{m,s} \, {\rm B}_{y_0} +   {\rm C}_{y_0} \mathcal{K}_{e,s} \big) \notag \\&
-\dfrac{h^3}{3} \mathrm{ H}\,\mathcal{W}_{\mathrm{shell}}  \big(  \mathcal{E}_{m,s} ,
\mathcal{E}_{m,s}{\rm B}_{y_0}+{\rm C}_{y_0}\, \mathcal{K}_{e,s} \big)+
\dfrac{h^3}{6}\, \mathcal{W}_{\mathrm{shell}}  \big(  \mathcal{E}_{m,s} ,
( \mathcal{E}_{m,s}{\rm B}_{y_0}+{\rm C}_{y_0}\, \mathcal{K}_{e,s}){\rm B}_{y_0} \big)\notag\vspace{2.5mm}\\
&+  \Big(h-{\rm K}\,\dfrac{h^3}{12}\Big)\,
W_{\mathrm{curv}}\big(  \mathcal{K}_{e,s} \big)    +  \dfrac{h^3}{12}
W_{\mathrm{curv}}\big(  \mathcal{K}_{e,s}   {\rm B}_{y_0} \,  \big).
\end{align}
\begin{proposition}\label{coerh3r}{\rm [Coercivity  in the truncated theory including terms up to order $O(h^3)$]} Assume that the constitutive coefficients are  such that $\mu>0$, $2\,\lambda+\mu> 0$, $b_1>0$, $b_2>0$, $b_3>0$ and $L_{\rm c}>0$ and let $c_2^+$  denote the smallest eigenvalue  of
	$
	W_{\mathrm{curv}}(  S ),
	$
	and $c_1^+$ and $ C_1^+>0$ denote the smallest and the largest eigenvalues of the quadratic form $W_{\mathrm{shell}}^\infty(  S)$.
	If the thickness $h$ satisfies  one of the following conditions:
	\begin{enumerate}\label{fcondh3}
		\item[i)] $	h\max\{\sup_{x\in\omega}|\kappa_1|, \sup_{x\in\omega}|\kappa_2|\}<\alpha \text{ and }	h^2<\frac{(5-2\sqrt{6})(\alpha^2-12)^2}{4\, \alpha^2}\frac{ {c_2^+}}{\max\{C_1^+,\mu_{\rm c}\}} \text{ with }  0<\alpha<2\sqrt{3}$;
		\item[ii)] $h\max\{\sup_{x\in\omega}|\kappa_1|, \sup_{x\in\omega}|\kappa_2|\}<\frac{1}{a} \text{ with }  a>\max\Big\{1 + \frac{\sqrt{2}}{2},\frac{1+\sqrt{1+3\frac{\max\{C_1^+,\mu_{\rm c}\} }{\min\{c_1^+,\mu_{\rm c}\} }}}{2}\Big\}$,
	\end{enumerate} then
	the total energy density $W^{(h^3)}(\mathcal{E}_{m,s}, \mathcal{K}_{e,s})$
	is coercive, in the sense that  there exists   a constant $a_1^+>0$ such that 
	$	W^{(h^3)}(\mathcal{E}_{m,s}, \mathcal{K}_{e,s})\,\geq\, a_1^+\, \big(  \lVert \mathcal{E}_{m,s}\rVert ^2 + \lVert \mathcal{K}_{e,s}\rVert ^2\,\big) ,
	$ where
	$a_1^+$ depends on the constitutive coefficients. The following inequality holds true
	\begin{align}
	 W^{(h^3)}(\mathcal{E}_{m,s}, \mathcal{K}_{e,s})\,\geq\, a_1^+\, \big( \lVert \mathcal{E}_{m,s}\rVert ^2 +\lVert
	\mathcal{E}_{m,s}{\rm B}_{y_0}+{\rm C}_{y_0}\, \mathcal{K}_{e,s}  \rVert ^2+ \lVert \mathcal{K}_{e,s}\rVert ^2\,\big).
	\end{align}
	\hfill$\blacksquare$
\end{proposition}

We define the lifted quantities $\widehat{\rm I}_{y_0} \in \mathbb{R}^{3\times 3}$ by
$
\label{first_fundamental_form_lift1}
\widehat{\rm I}_{y_0}\:\,=\,({\nabla  y_0}|n_0)^T({\nabla  y_0}|n_0)\,= {\rm I}_{y_0}^\flat+\widehat{0}_3\,$ and $\widehat{\rm II}_{y_0}\in\mathbb{R}^{3\times3}$ by
$
\widehat{\rm II}_{y_0}\,=\,-(\nabla y_0|n_0)^T(\nabla n_0|n_0)\,= \,{\rm II}_{y_0}^\flat-\widehat{0}_3\, .$
Some useful  properties of the tensors involved in the variational formulation of the considered shell models \cite{GhibaNeffPartI,GhibaNeffPartII} are gathered next:

\newpage
\begin{remark}\label{propAB}{\rm }
	The following identities are satisfied :
	\begin{itemize}
		\item [i)] 
		
		$\tr[{\rm A}_{y_0}]\,=\,2,$ \qquad	${\det}[{\rm A}_{y_0}]\,=\,0;\qquad $
		$\tr[{\rm B}_{y_0}]\,=\,2\,{\rm H}\,$,\qquad  ${\det}[{\rm B}_{y_0}]\,=\,0,$ \\
		${\rm A}_{y_0} = [\nabla\Theta \,]^{-T}\; {\rm I}_{y_0}^\flat \; [\nabla\Theta \,]^{-1}=\,\id_3-(0|0|\nabla\Theta \,.e_3)\,[	\nabla\Theta \,]^{-1}\,=\,\id_3-(0|0|n_0)\,(0|0|n_0)^T$, \\ $
		{\rm B}_{y_0} = [\nabla\Theta \,]^{-T}\; {\rm II}_{y_0}^\flat \; [\nabla\Theta \,]^{-1}$;
		
		\item[ii)] ${\rm B}_{y_0}$ satisfies the equation of Cayley-Hamilton type
		$
		{\rm B}_{y_0}^2-2\,{\rm H}\, {\rm B}_{y_0}+{\rm K}\, {\rm A}_{y_0}\,=\,0_3;
		$
		\item[iii)] ${\rm A}_{y_0}{\rm B}_{y_0}\,=\,{\rm B}_{y_0}{\rm A}_{y_0}\,=\,{\rm B}_{y_0}$, \quad  ${\rm A}_{y_0}^2\,=\,{\rm A}_{y_0}$, \quad  ${\rm C}_{y_0}\in \mathfrak{so}(3)$, $\quad {\rm C}_{y_0}^2\,=\,-{\rm A}_{y_0}$, \quad $\lVert {\rm C}_{y_0}\rVert ^2=2$;
		\item[iv)] $
		\overline{Q}_{e,s}^T\,(\nabla [\overline{Q}_{e,s}\nabla\Theta \,.e_3]\,|\,0)\,[\nabla\Theta \,]^{-1}\,\,=\,\,{\rm C}_{y_0} \mathcal{K}_{e,s}-{\rm B}_{y_0};
		$
		\item[v)] ${\rm C}_{y_0} \mathcal{K}_{e,s} {\rm A}_{y_0}\,\,=\,\,{\rm C}_{y_0} \mathcal{K}_{e,s} $,\quad  $\mathcal{E}_{m,s} {\rm A}_{y_0}\,\,=\,\,\mathcal{E}_{m,s} $.
	\end{itemize}
\end{remark}
 Further,  we can express the strain tensors using the (referential) fundamental forms $ {\rm I}_{y_0} $, $ {\rm II}_{y_0}$ and $ {\rm L}_{y_0} $  (instead of using the matrices $ {\rm A}_{y_0}$, $ {\rm B}_{y_0}$ and  $ {\rm C}_{y_0}$), see \cite{GhibaNeffPartI,GhibaNeffPartIII}, i.e.,
\begin{align}\label{eq5}
\mathcal{E}_{m,s}=&\quad\ \, [\nabla\Theta \,]^{-T}
\begin{footnotesize}\left( \begin{array}{c|c}
(\overline{Q}_{e,s} \nabla y_0)^{T} \nabla m- {\rm I}_{y_0} & 0 \vspace{4pt}\\
(\overline{Q}_{e,s}  n_0)^{T} \nabla m & 0
\end{array} \right)\end{footnotesize} [\nabla\Theta \,]^{-1}=
[\nabla\Theta \,]^{-T}
\begin{footnotesize}\left( \begin{array}{c|c}
\mathcal{G} & 0 \vspace{4pt}\\
\mathcal{T}  & 0
\end{array} \right)\end{footnotesize} [\nabla\Theta \,]^{-1},
\vspace{6pt}\\
\mathrm{C}_{y_0} \mathcal{K}_{e,s} = &\quad\ \, [\nabla\Theta \,]^{-T}
\begin{footnotesize}\left( \begin{array}{c|c}
(\overline{Q}_{e,s} \nabla y_0)^{T} \nabla (\overline{Q}_{e,s} n_0)+ {\rm II}_{y_0} & 0 \vspace{4pt}\notag\\
0 & 0
\end{array} \right)\end{footnotesize} [\nabla\Theta \,]^{-1}= -[\nabla\Theta \,]^{-T}
\begin{footnotesize}\left( \begin{array}{c|c}
\mathcal{R} & 0 \vspace{4pt}\notag\\
0 & 0
\end{array} \right)\end{footnotesize} [\nabla\Theta \,]^{-1},\notag\\
\mathcal{E}_{m,s} {\rm B}_{y_0}^2  + \mathrm{C}_{y_0} \mathcal{K}_{e,s} {\rm B}_{y_0}
= &\, -[\nabla\Theta \,]^{-T}
\begin{footnotesize}\left( \begin{array}{c|c}
 (\mathcal{R} -\mathcal{G} \,{\rm L}_{y_0})\,{\rm L}_{y_0}& 0 \vspace{4pt}\\
\mathcal{T} \,{\rm L}_{y_0}^2 & 0
\end{array} \right)\end{footnotesize} [\nabla\Theta \,]^{-1}\notag
,
\end{align}
where
\begin{align}\label{eq41}
\mathcal{G} :=&\, (\overline{Q}_{e,s} \nabla y_0)^{T} \nabla m- {\rm I}_{y_0}\not\in {\rm Sym}(2)\qquad\qquad\qquad\qquad\quad\ \ \ \,\  \,\textrm{\it the change of metric tensor},
\\
\mathcal{T}:=& \, (\overline{Q}_{e,s}  n_0)^{T} \nabla m= \, \left(\bigl\langle\overline{Q}_{e,s}  n_0, \partial_{x_1} m\bigr\rangle,\bigl\langle\overline{Q}_{e,s}  n_0, \partial_{x_2} m\bigr\rangle\right)\qquad  \textrm{\it the transverse shear deformation (row) vector},\notag
\\
\mathcal{R} :=& \, -(\overline{Q}_{e,s} \nabla y_0)^{T} \nabla (\overline{Q}_{e,s} n_0)- {\rm II}_{y_0}\not\in {\rm Sym}(2)
\quad  \qquad\qquad \ \ \, \,\,\textrm{\it the bending  strain tensor}.
\notag\end{align}
In the above,  we can replace $ \;{\rm L}^2_{y_0} = 2\,{\rm H}\,{\rm L}_{y_0} - {\rm K}\,\id_2\; $ by the Cayley-Hamilton theorem. The nonsymmetric quantity $ \mathcal{R}-\mathcal{G} \,{\rm L}_{y_0}$ represents {\it the change of curvature} tensor. The choice of this name will be justified  in \cite{GhibaNeffPartVI}. For now, we just mention that 
the definition of $\mathcal{G}$ is related to the classical {\it change of metric }  tensor in the Koiter model \cite{Steigmann12,Steigmann13,Ciarlet00}
\begin{align} \mathcal{G}_{\mathrm{Koiter}} := \dfrac12 \big[ ( \nabla m)^{T} \nabla m- {\rm I}_{y_0}\big]=\dfrac12\,({\rm I}_m-{\rm I}_{y_0}),\end{align}
while the bending strain tensor may be compared  with the classical {\it bending strain tensor}  in the Koiter model
\begin{align} \mathcal{R}_{\rm{Koiter}} :=  -(\nabla m)^{T} \nabla n- {\rm II}_{y_0}.\end{align}

\section{Linearized   Cosserat shell model including terms up to order $O(h^5)$}\setcounter{equation}{0}
In this section we develop the linearisation   for the elastic Cosserat shell model including terms up to order $O(h^5)$, i.e., for situations of small Cosserat midsurface deformations and small change of curvature. 
\subsection{Linearized strain measures in  the Cosserat shell model}

We express the total midsurface deformation 
\begin{align}
m(x_1,x_2)=y_0(x_1,x_2)+v(x_1,x_2),
\end{align}
with $v:\omega\to \mathbb{R}^3$, the infinitesimal shell-midsurface displacement. For the elastic rotation tensor $ \overline{Q}_{e,s}\in\rm{SO}(3) $ there is a skew-symmetric matrix  \begin{align}
\overline{A}_\vartheta:={\rm Anti}(\vartheta_1,\vartheta_2,\vartheta_3):=\begin{footnotesize}
\begin{pmatrix}
0&-\vartheta_3&\vartheta_2\\
\vartheta_3&0&-\vartheta_1\\
-\vartheta_2&\vartheta_1&0
\end{pmatrix}\end{footnotesize}\in \mathfrak{so}(3), \quad \qquad {\rm Anti}:\mathbb{R}^3\to \mathfrak{so}(3),
\end{align}
where $ \vartheta={\rm axl}( \overline{A}_\vartheta) $ denotes the axial vector of $ \overline{A}_\vartheta $, such that 
$ \overline{Q}_{e,s}:=\exp(\overline{A}_\vartheta)\;= \;\sum_{k=0}^{\infty} \frac{1}{k!} \,\overline{A}_\vartheta^k\; = \;\id_3 + \overline{A}_\vartheta+\textrm{h.o.t.}$ 
The tensor field $\overline{A}_\vartheta$ is   the infinitesimal elastic  microrotation. Here, ``h.o.t" stands for terms of order higher than linear with respect to $v$ and $\overline{A}_\vartheta$.  Then, we can expand
\begin{equation}\label{equ1}
\overline{Q}_{e,s}^T\nabla m-\nabla y_0 = (\id_3 +\overline{A}_\vartheta^T+\textrm{h.o.t.} )(\nabla v + \nabla y_0) -\nabla y_0 
=\nabla v - \overline{A}_\vartheta\nabla y_0+\textrm{h.o.t.}.
\end{equation}
Since $\nabla y_0$ is given, the first order term in \eqref{equ1} is linear in $v$ and $ \overline{A}_\vartheta$. Correspondingly,  we get from the non-symmetric  \textit{shell strain tensor } (which characterises both the in-plane deformation and the transverse shear deformation)
\begin{equation*} 
\mathcal{E}_{m,s} = ( \overline{Q}_{e,s}^T\nabla m-\nabla y_0\; |\; 0)\; [\nabla\Theta \,]^{-1}
\end{equation*}
the linearization
\begin{equation}\label{equ2}
\begin{array}{rcl}
\mathcal{E}_{m,s}^{\rm{lin}} &=&   ( \nabla v - \overline{A}_\vartheta\nabla y_0\; |\; 0)\; [\nabla\Theta \,]^{-1}
=(\partial_{x_1} v - \vartheta\times a_1\;|\;  \partial_{x_2} v - \vartheta\times a_2 \;|\; 0)\; [\nabla\Theta \,]^{-1}\;\not\in {\rm Sym}(3).
\end{array}
\end{equation}

Our aim now is to express all the deformation measures and the linearised models in terms of $\overline{A}_\vartheta$ as well as in terms of its axial vector $\vartheta$. The following definitions are used to express  these quantities in terms of $\vartheta$.

For any column vector $ q\in \mathbb{R}^3$ and any  matrix $ M=(M_1|M_2|M_3)\in \mathbb{R}^{3\times 3}  $ we define the cross-product
\begin{align} \label{ec1} 
q\times M :=&\, (q\times M_1\,|\,q\times M_2\,|\,q\times M_3) \qquad\qquad(\mbox{operates on columns})\quad \textrm{and}
\vspace{6pt}\\
M^T \times q^T:=&\, - (q\times M)^T \qquad\qquad \qquad\qquad  \qquad(\mbox{operates on rows}).\notag
\end{align}
Note that $ M $ can also be a $ 3\times 2 $ matrix, the definition remains the same.
Let us note some properties of these operations: for any column vectors $ q_1, q_2\in \mathbb{R}^{3}  $ and any   matrices $ M , N \in \mathbb{R}^{3\times 3} $ (or $\in \mathbb{R}^{3\times 2} $) we have
\begin{align} \label{ec2} 
(q_1\times M )\,q_2 = &\, \,q_1\times (M \,    q_2)\,,
\qquad 
q_1^T(q_2\times M ) =  \,(q_1\times q_2)^T\, M =  (q_1^T\times  q_2^T)\,M =  -q_2^T(q_1\times M ) 
\end{align}
and, more general
\begin{align} \label{ec3} 
(q_1\times M )N = &\, q_1\,\times (M\, N)\,,
\qquad 
N^T(q_2\times M ) = \, -(q_2\times N)^T\, M  = (N^T\times  q_2^T)\,M\,.
\end{align}
With these relations, the infinitesimal microrotation $\overline{A}_\vartheta$ can be expressed as 
\begin{equation} \label{ec4} 
\overline{A}_\vartheta :=  \vartheta\times \id_3 = \id_3\times \vartheta^T\in \mathbb{R}^{3\times 3}.
\end{equation}

It is possible to  write the linearised strain measures in a simplified form. The relation $ \mathcal{E}_{m,s}^{\rm{lin}}=  ( \nabla v - \overline{A}_\vartheta\nabla y_0\; |\; 0)\; [\nabla\Theta \,]^{-1} $
turns into
\begin{equation}\label{ec10}
\begin{array}{rcl}
\mathcal{E}_{m,s}^{\rm{lin}} &=&   ( \nabla v -  \vartheta\times\nabla y_0\; |\; 0)\; [\nabla\Theta \,]^{-1} = \big[( \nabla v\,|\,0) -  \vartheta\times(\nabla y_0\,|\,0) \big]\; [\nabla\Theta \,]^{-1}.
\end{array}
\end{equation}

Doing the same for the elastic  \textit{shell bending-curvature tensor}
\begin{align} \mathcal{K}_{e,s} := \Big(\mbox{axl}(\overline{Q}_{e,s}^T\partial_{x_1}\overline{Q}_{e,s})\,|\, \mbox{axl}(\overline{Q}_{e,s}^T\partial_{x_2}\overline{Q}_{e,s})\,|\, 0 \Big) \; [\nabla\Theta \,]^{-1}\end{align} 
since
\begin{equation} 
\overline{Q}_{e,s}^T\partial_{x_\alpha}\overline{Q}_{e,s}= (\id_3 - \overline{A}_\vartheta )\,\,\partial_{x_\alpha}\overline{A}_\vartheta+\textrm{h.o.t.}
= \partial_{x_\alpha}\overline{A}_\vartheta +\textrm{h.o.t.} = \underbrace{\overline{A}_{\partial_{x_\alpha}\vartheta}}_{\equiv\, {\rm Anti} \,\partial_{x_\alpha}\vartheta\,=\,\partial_{x_\alpha} {\rm Anti} \, \vartheta}+\,\textrm{h.o.t.}\;,
\end{equation}
i.e.,
$
\mbox{axl}\big(\overline{Q}_{e,s}^T\partial_{x_\alpha}\overline{Q}_{e,s}\big)= \partial_{x_\alpha} \vartheta+\textrm{h.o.t.}
$
we deduce
$
\mathcal{K}_{e,s}^{\rm{lin}} \,\,= \,\, (\mbox{axl}\big(\partial_{x_1}\overline{A}_\vartheta\big) \,|\, \mbox{axl}\big(\partial_{x_2}\overline{A}_\vartheta\big)\, |\, 0) \; [\nabla\Theta \,]^{-1}
$
and 
\begin{equation}\label{equ7}
\mathcal{K}_{e,s}^{\rm{lin}} \,\,= \,\, (\partial_{x_1}\vartheta \,|\, \partial_{x_2}\vartheta\, |\, 0) \; [\nabla\Theta \,]^{-1} \,\,= \,\, (\nabla\vartheta\, |\, 0) \; [\nabla\Theta \,]^{-1}\;.
\end{equation}
\newpage

\begin{remark} The matrix $ \mathrm{C}_{y_0} $ admits the form
	\begin{equation} \label{ec5} 
	\mathrm{C}_{y_0} = - n_0\times \id_3 = -n_0\times \mathrm{A}_{y_0}\;.
	\end{equation}
\end{remark}
\begin{proof}
	To prove \eqref{ec5} we put the definition (2.6) in the following form
	\begin{equation} \label{defC2}
	\mathrm{C}_{y_0} = [\nabla\Theta \,]^{-T}\; {\rm C}^\flat \; [\nabla\Theta \,]^{-1},
	\qquad \text{with}\qquad 
	{\rm C}=\begin{footnotesize}\begin{pmatrix}
	0 & \sqrt{\det {\rm I}_{y_0}} \\
	-\sqrt{\det {\rm I}_{y_0}} & 0 \\
	\end{pmatrix}\end{footnotesize}. 
	\end{equation}

	Here, and in the rest of the paper, $a_1, a_2, a_3\,$ denote the columns of $\nabla\Theta \,$, while  $a^1, a^2, a^3\,$ denote the rows of $[\nabla\Theta \,]^{-1}$, i.e. 
	\begin{align}
	\nabla\Theta \,&\,=\,(\nabla y_0|\,n_0)\,=\,(a_1|\,a_2|\,a_3),\qquad \qquad \quad  [\nabla\Theta \,]^{-1}\,=\,(a^1|\,a^2|\,a^3)^T.
	\end{align}
	In fact,   $a_1, a_2\,$ are the covariant base vectors and $ a^1, a^2\,$ are the contravariant base vectors in the tangent plane given by
	$
	a_\alpha:=\,\partial_{x_\alpha}y_{0},\  \langle a^\beta, a_\alpha\rangle \,=\,\delta_\alpha^\beta,\ \alpha,\beta=1,2,
	$
	and $a_3\,= a^3=n_0\,$. The following relations hold \cite[page 95]{Ciarlet00}: 
$
	\lVert a_1\times a_2\rVert=\sqrt{\det {\rm I}_{y_0}}, $ $  a_3\times a_1=\sqrt{\det {\rm I}_{y_0}}\, a^2, $ $ a_2\times a_3=\sqrt{\det {\rm I}_{y_0}}\, a^1.
$

	Next, we calculate
	\begin{align} \label{ec7} 
	[\nabla\Theta \,]^{-T}\; {\rm C}^\flat &=
	(a^1\,|\,a^2\,|\,a^3)\begin{footnotesize}
	\begin{pmatrix}
	0 & \sqrt{\det {\rm I}_{y_0}} & 0\\
	-\sqrt{\det {\rm I}_{y_0}} & 0 & 0\\
	0 & 0 & 0
	\end{pmatrix}\end{footnotesize}
	= (-a^2\sqrt{\det {\rm I}_{y_0}}\,|\,a^1\sqrt{\det {\rm I}_{y_0}}\,|\,0)
	\vspace{6pt}\\
	&=  (a_1\times n_0 \,|\, a_2\times n_0 \,|\, 0 ) =  -n_0\times \nabla\Theta \,.\notag
	\end{align}
	Using  \eqref{ec2} and the decomposition $ \id_3 = {\rm A}_{y_0} + (0|0|n_0) \,(0|0|n_0)^T$, see Remark \ref{propAB}, we have
	\[  
	\mathrm{C}_{y_0} = ( -n_0\times \nabla\Theta \, )[\nabla\Theta \,]^{-1} = - n_0\times \id_3 = -n_0\times \mathrm{A}_{y_0}
	\]
	and  \eqref{ec5} is proved.
\end{proof}

From   the definition $ \;\mathcal{K}_{e,s}^{\rm{lin}}= (\nabla\vartheta\, |\, 0) \; [\nabla\Theta \,]^{-1}\; $ together with \eqref{ec5} we obtain
\begin{equation}\label{ec11}
\begin{array}{rcl}
\mathrm{C}_{y_0}\,\mathcal{K}_{e,s}^{\rm{lin}} &=& - n_0\times  (\nabla\vartheta\, |\, 0) \; [\nabla\Theta \,]^{-1}\,.
\end{array}
\end{equation}

In conclusion, the linear strain measures $ \mathcal{E}_{m,s}^{\rm{lin}} $ and $ \mathcal{K}_{e,s}^{\rm{lin}} $ are given by \eqref{equ2} and \eqref{equ7}.

In order to  obtain a comparison with the classical linear Koiter-shell model,  we also deduce the linear approximation of the constitutive variables $ \mathcal{G}$, $ \mathcal{T}$ and $ \mathcal{R}$.
The linear approximation  of the {\it change of metric tensor} from \eqref{eq41} \begin{align} \mathcal{G}^{\rm{lin}} =  (\nabla y_0)^{T} ( \nabla v - \overline{A}_\vartheta\nabla y_0 ) ,\end{align}
 admits the alternative expression, due to \eqref{ec3}, 
\begin{equation}\label{ec12}
\begin{array}{rcl}
\mathcal{G}^{\rm{lin}} &=&    (\nabla y_0)^{T} ( \nabla v - \vartheta\times\nabla y_0 )\;,\qquad\mbox{or}
\qquad
\mathcal{G}^{\rm{lin}}  = (\nabla y_0)^{T}  \nabla v + (\vartheta\times\nabla y_0 )^T\nabla y_0\,.
\end{array}
\end{equation}
Using \eqref{ec2},  the linear approximation of the {\it transverse shear vector} \begin{align}\mathcal{T}^{\rm{lin}}=  n_0^T ( \nabla v - \overline{A}_\vartheta\nabla y_0) \end{align} reads
\begin{equation}\label{ec13}
\begin{array}{rcl}
\mathcal{T}^{\rm{lin}} &=&   n_0^{T} ( \nabla v - \vartheta\times\nabla y_0 )\;,\qquad\mbox{or}
\qquad 
\mathcal{T}^{\rm{lin}}  = n_0^{T}  \nabla v + \vartheta^T(n_0\times\nabla y_0 ) \;=\;  
n_0^{T}  \nabla v + (\vartheta\times n_0)^T\nabla y_0 \;.
\end{array}
\end{equation}
Using \eqref{ec12} and \eqref{ec13} we can form the matrix
\begin{align}\label{ec14}
\begin{footnotesize}\begin{pmatrix} 
\mathcal{G}^{\rm{lin}}
\\
\mathcal{T}^{\rm{lin}}
\end{pmatrix} \end{footnotesize}=&   \begin{footnotesize}\begin{pmatrix} 
(\nabla y_0)^{T}  \nabla v + (\vartheta\times\nabla y_0 )^T\nabla y_0 
\vspace{4pt}\\
n_0^{T}  \nabla v + (\vartheta\times n_0)^T\nabla y_0 
\end{pmatrix} \end{footnotesize}
= [\nabla\Theta \,]^{T}  \nabla v + [\,\vartheta\times\nabla\Theta \, ]^T\nabla y_0 
=\,\, [\nabla\Theta \,]^{T} ( \nabla v -\vartheta\times \nabla y_0)  \;,
\end{align}
in accordance with \eqref{ec10} and \eqref{eq5}.

For the {\it bending strain tensor} we get the linearisation 
\begin{align}  \mathcal{R}^{\rm{lin}} &= - (\nabla y_0)^{T} (\id_3 - \overline{A}_\vartheta)\nabla [(\id_3 + \overline{A}_\vartheta) n_0]- {\rm II}_{y_0}= (\nabla y_0)^{T} (\overline{A}_\vartheta \nabla n_0 -\nabla(\overline{A}_\vartheta  n_0 ))\\&=-(\nabla y_0)^{T} \nabla\overline{A}_\vartheta\,  n_0 =  - (\nabla y_0)^{T} ( \partial_{x_1}\vartheta\times n_0\;|\; \partial_{x_2}\vartheta\times n_0 )\notag
\end{align} and, see \eqref{ec3}, its alternative expressions 
\begin{equation}\label{ec15}
\begin{array}{rcl}
\mathcal{R}^{\rm{lin}} &=&    (\nabla y_0)^{T} ( n_0\times\nabla \vartheta )\;\qquad\mbox{or}\qquad
\mathcal{R}^{\rm{lin}}  = - ( n_0\times\nabla y_0)^{T}  \nabla \vartheta \;.
\end{array}
\end{equation}

To  obtain a comparison with the classical linear Koiter-shell model, let us first present an alternative form of $\mathcal{G}^{\rm lin}$.
From \eqref{equ1} we have
\begin{equation}\label{equ3}
\begin{array}{rcl}
\mathcal{G}^{\rm lin} &=&    (\nabla y_0)^{T} ( \nabla v - \overline{A}_\vartheta\nabla y_0 ) \;=\; (\nabla y_0)^{T} ( \partial_{x_1} u + a_1\times \vartheta\;|\;  \partial_{x_2} u + a_2\times \vartheta) 
\vspace{6pt}\\
&=&  (\nabla y_0)^{T}(\nabla v) + (a_1\,|\,a_2)^T(  a_1\times \vartheta\;|\;    a_2\times \vartheta)\vspace{1mm}= (\nabla y_0)^{T}(\nabla v) + \begin{footnotesize}
\begin{pmatrix}
0 & \bigl\langle \vartheta, a_1\times a_2\bigr\rangle
\\
-\bigl\langle \vartheta, a_1\times a_2\bigr\rangle & 0
\end{pmatrix}\end{footnotesize}]
\vspace{6pt}\\
&=& (\nabla y_0)^{T}(\nabla v) + \bigl\langle\vartheta, n_0\bigr\rangle\begin{footnotesize}
\begin{pmatrix}
0 & \sqrt{\det {\rm I}_{y_0}}
\\
-\sqrt{\det {\rm I}_{y_0}}  & 0
\end{pmatrix}\end{footnotesize} =   (\nabla y_0)^{T}(\nabla v) + \bigl\langle\vartheta, n_0\bigr\rangle\, {\rm C}.
\end{array}
\end{equation}
Remember  that ${\rm C}$ defined by \eqref{defC2} is 
a $ 2\times 2  $ skew-symmetric matrix.  From \eqref{equ3}  we note the relation
\begin{equation}
\label{equ12,5}
\sym(\mathcal{G}^{\rm{lin}} )\; =\; \sym\big[ (\nabla y_0)^{T}(\nabla v)\big] \;=:\; \mathcal{G}_{\rm{Koiter}}^{\rm{lin}} \;,
\end{equation}
therefore  $ \;\mathcal{G}_{\rm{Koiter}}^{\rm{lin}} $ corresponds to the symmetric part of our $ \mathcal{G}^{\rm{lin}} $.

\subsection{The variational problem of the linearised $O(h^5)$-Cosserat shell model}

The form of the energy density remains unchanged upon linearization, since the geometrically nonlinear model is physically linear (quadratic in the employed strain and curvature measures). Thus, 
the  variational problem for the linear Cosserat $O(h^5)$-shell model  is  to find a midsurface displacement vector field 
$v:\omega\subset\mathbb{R}^2\to\mathbb{R}^3$ and the micro-rotation vector field $\vartheta:\omega\subset\mathbb{R}^2\to\mathbb{R}^3$ minimizing on $\omega$:
\begin{align}\label{e89l}
I(v,\vartheta)\!=\!\! \int_{\omega}   \!\!\Big[  &\Big(h+{\rm K}\,\dfrac{h^3}{12}\Big)\,
W_{\mathrm{shell}}\big(    \mathcal{E}_{m,s}^{\rm{lin}}  \big)+  \Big(\dfrac{h^3}{12}\,-{\rm K}\,\dfrac{h^5}{80}\Big)\,
W_{\mathrm{shell}}  \big(   \mathcal{E}_{m,s}^{\rm{lin}}  \, {\rm B}_{y_0} +   {\rm C}_{y_0} \mathcal{K}_{e,s}^{\rm{lin}}  \big)  \notag\\&
-\dfrac{h^3}{3} \mathrm{ H}\,\mathcal{W}_{\mathrm{shell}}  \big(  \mathcal{E}_{m,s}^{\rm{lin}}  ,
\mathcal{E}_{m,s}^{\rm{lin}} {\rm B}_{y_0}+{\rm C}_{y_0}\, \mathcal{K}_{e,s}^{\rm{lin}}  \big)+
\dfrac{h^3}{6}\, \mathcal{W}_{\mathrm{shell}}  \big(  \mathcal{E}_{m,s}^{\rm{lin}}  ,
( \mathcal{E}_{m,s}^{\rm{lin}} {\rm B}_{y_0}+{\rm C}_{y_0}\, \mathcal{K}_{e,s}^{\rm{lin}} ){\rm B}_{y_0} \big)\notag\\&+ \,\dfrac{h^5}{80}\,\,
W_{\mathrm{mp}} \big((  \mathcal{E}_{m,s}^{\rm{lin}}  \, {\rm B}_{y_0} +  {\rm C}_{y_0} \mathcal{K}_{e,s}^{\rm{lin}}  )   {\rm B}_{y_0} \,\big) \\&+ \,\Big(h-{\rm K}\,\dfrac{h^3}{12}\Big)\,
W_{\mathrm{curv}}\big(  \mathcal{K}_{e,s}^{\rm{lin}}  \big)    +  \Big(\dfrac{h^3}{12}\,-{\rm K}\,\dfrac{h^5}{80}\Big)\,
W_{\mathrm{curv}}\big(  \mathcal{K}_{e,s} ^{\rm{lin}}   {\rm B}_{y_0} \,  \big)  + \,\dfrac{h^5}{80}\,\,
W_{\mathrm{curv}}\big(  \mathcal{K}_{e,s}^{\rm{lin}}    {\rm B}_{y_0}^2  \big)
\Big] {\rm det}(\nabla y_0|n_0)     d a\notag\\ &- {\Pi}^{\rm lin}(v,\vartheta)\,,\notag
\end{align}
where
\begin{align}\label{quadraticformsl}
\mathcal{E}_{m,s}^{\rm{lin}} &:= ( \nabla v - \overline{A}_\vartheta\nabla y_0\; |\; 0)\; [\nabla\Theta \,]^{-1}=\big[( \nabla v\,|\,0) -  \vartheta\times(\nabla y_0\,|\,0) \big]\; [\nabla\Theta \,]^{-1},\notag\\
\mathcal{K}_{e,s}^{\rm{lin}} &:= \Big(\mbox{axl}(\partial_{x_1}\overline{A}_\vartheta)\,|\, \mbox{axl}(\partial_{x_2}\overline{A}_\vartheta)\,|\, 0 \Big) \; [\nabla\Theta \,]^{-1} = \, (\nabla\vartheta\, |\, 0) \; [\nabla\Theta \,]^{-1},
\end{align}
and $\overline{\Pi}^{\rm lin}(v,\vartheta)$ is the linearization of $\overline{\Pi}(m,\overline{Q}_{e,s})$.

For simplicity, we consider only ($\gamma_d=\partial \omega$) the Dirichlet-type homogeneous boundary conditions: \begin{align}\label{boundary2}
v\big|_{\partial \Omega}&=0 \qquad \text{simply supported (fixed, welded)}, \qquad \quad 
\vartheta\big|_{\partial \Omega}=0\ \ (\overline{A}_\vartheta\big|_{\partial \Omega}=0),\ \  \ \ \ \ \ \ \text{clamped},
\end{align}
where the boundary conditions are to be understood in the sense of traces. Therefore,  the admissible set $\mathcal{A}_{\rm lin}$ of solutions $(v,\vartheta)$ is defined by
\begin{equation}\label{21lin}
(v,\vartheta)\in \mathcal{A}_{\rm lin}:={\rm H}^1_0(\omega, \mathbb{R}^3)\times{\rm H}^1_0(\omega, \mathbb{R}^3).
\end{equation}

\subsection{Existence for the linearised  Cosserat shell model}
\subsubsection{Existence result for the linearised $O(h^5)$-Cosserat shell model}
\allowdisplaybreaks
We rewrite the minimization problem \eqref{e89l} in a weak form. For this we define the  operators
\begin{align}\label{operators}
\mathcal{E}^{\rm{lin}}&:\mathcal{A}_{\rm lin}\to \mathbb{R}^{3\times 3},  \qquad \mathcal{E}^{\rm{lin}}(v,\vartheta)\,= \big[( \nabla v\,|\,0) -  \vartheta\times(\nabla y_0\,|\,0) \big]\; [\nabla\Theta \,]^{-1},\notag\\
\mathcal{K}^{\rm{lin}}&:\mathcal{A}_{\rm lin}\to \mathbb{R}^{3\times 3},  \qquad\mathcal{K}^{\rm{lin}} (v,\vartheta)= (\nabla\vartheta\, |\, 0) \; [\nabla\Theta \,]^{-1},
\end{align}
the bilinear form
$
\mathcal{B}^{\rm{lin}}:\mathcal{A}_{\rm lin}\times \mathcal{A}_{\rm lin}\to \mathbb{R},$
\begin{align}
\mathcal{B}^{\rm{lin}}((v,\vartheta),(\widetilde{v},\widetilde{\vartheta})):=\int_{\omega}   \!\!\Big[  &\Big(h+{\rm K}\,\dfrac{h^3}{12}\Big)\,
\mathcal{W}_{\mathrm{shell}}\big(    \mathcal{E}^{\rm{lin}}(v,\vartheta), \mathcal{E}^{\rm{lin}}(\widetilde{v},\widetilde{\vartheta})  \big)\notag\\&+  \Big(\dfrac{h^3}{12}\,-{\rm K}\,\dfrac{h^5}{80}\Big)\,
\mathcal{W}_{\mathrm{shell}}  \big(   \mathcal{E}^{\rm{lin}}(v,\vartheta)  \, {\rm B}_{y_0} +   {\rm C}_{y_0} \mathcal{K}^{\rm{lin}}(v,\vartheta), \mathcal{E}^{\rm{lin}}(\widetilde{v},\widetilde{\vartheta})  \, {\rm B}_{y_0} +   {\rm C}_{y_0} \mathcal{K}^{\rm{lin}}(\widetilde{v},\widetilde{\vartheta})  \big)  \notag\\&
-\dfrac{h^3}{6} \mathrm{ H}\,\mathcal{W}_{\mathrm{shell}}  \big(  \mathcal{E}^{\rm{lin}}(v,\vartheta)  ,
\mathcal{E}^{\rm{lin}}(\widetilde{v},\widetilde{\vartheta}) {\rm B}_{y_0}+{\rm C}_{y_0}\, \mathcal{K}^{\rm{lin}}(\widetilde{v},\widetilde{\vartheta})  \big)\notag\\
&-\dfrac{h^3}{6} \mathrm{ H}\,\mathcal{W}_{\mathrm{shell}}  \big(  \mathcal{E}^{\rm{lin}}(\widetilde{v},\widetilde{\vartheta})  ,
\mathcal{E}^{\rm{lin}}(v,\vartheta) {\rm B}_{y_0}+{\rm C}_{y_0}\, \mathcal{K}^{\rm{lin}}(v,\vartheta) \big)\notag\\&+
\dfrac{h^3}{12}\, \mathcal{W}_{\mathrm{shell}}  \big(  \mathcal{E}^{\rm{lin}}(v,\vartheta)  ,
( \mathcal{E}^{\rm{lin}}(\widetilde{v},\widetilde{\vartheta}) {\rm B}_{y_0}+{\rm C}_{y_0}\, \mathcal{K}^{\rm{lin}}(\widetilde{v},\widetilde{\vartheta}) ){\rm B}_{y_0} \big)\notag
\\&+
\dfrac{h^3}{12}\, \mathcal{W}_{\mathrm{shell}}  \big(  \mathcal{E}^{\rm{lin}}(\widetilde{v},\widetilde{\vartheta})  ,
( \mathcal{E}^{\rm{lin}}(v,\vartheta) {\rm B}_{y_0}+{\rm C}_{y_0}\, \mathcal{K}^{\rm{lin}}(v,\vartheta) ){\rm B}_{y_0} \big)\\&+ \,\dfrac{h^5}{80}\,\,
\mathcal{W}_{\mathrm{mp}} \big((  \mathcal{E}^{\rm{lin}}(v,\vartheta)  \, {\rm B}_{y_0} +  {\rm C}_{y_0} \mathcal{K}^{\rm{lin}}(v,\vartheta)  )   {\rm B}_{y_0} ,(  \mathcal{E}^{\rm{lin}}(\widetilde{v},\widetilde{\vartheta})  \, {\rm B}_{y_0} +  {\rm C}_{y_0} \mathcal{K}^{\rm{lin}}(\widetilde{v},\widetilde{\vartheta})  )   {\rm B}_{y_0}\big)\notag \\&+ \,\Big(h-{\rm K}\,\dfrac{h^3}{12}\Big)\,
\mathcal{W}_{\mathrm{curv}}\big(  \mathcal{K}^{\rm{lin}}(v,\vartheta), \mathcal{K}^{\rm{lin}}(\widetilde{v},\widetilde{\vartheta})  \big)  \notag \\& +  \Big(\dfrac{h^3}{12}\,-{\rm K}\,\dfrac{h^5}{80}\Big)\,
\mathcal{W}_{\mathrm{curv}}\big(  \mathcal{K}^{\rm{lin}}(v,\vartheta)   {\rm B}_{y_0},\mathcal{K}^{\rm{lin}}(\widetilde{v},\widetilde{\vartheta})    {\rm B}_{y_0} \,  \big) \notag \\&+ \,\dfrac{h^5}{80}\,\,
\mathcal{W}_{\mathrm{curv}}\big(  \mathcal{K}^{\rm{lin}}(v,\vartheta)    {\rm B}_{y_0}^2,\mathcal{K}^{\rm{lin}}(\widetilde{v},\widetilde{\vartheta})      {\rm B}_{y_0}^2   \big)
\Big] {\rm det}(\nabla y_0|n_0)     d a\notag.
\end{align}
and the linear operator
\begin{align}
{\Pi}^{\rm lin}:\mathcal{A}_{\rm lin}\to\mathbb{R}, \qquad \qquad {\Pi}^{\rm lin}(\widetilde{v},\widetilde{\vartheta})=\overline{\Pi}^{\rm lin}(\widetilde{v},\widetilde{\vartheta}).
\end{align}
Then the weak form of the equilibrium problem of the linear theory of Cosserat shells including terms up to order $O(h^5)$ is to find
$(u,\vartheta)\in\mathcal{A}_{\rm lin}$ satisfying 
\begin{align}\label{wfproblem}
\mathcal{B}^{\rm{lin}}((v,\vartheta),(\widetilde{v},\widetilde{\vartheta}))={\Pi}^{\rm lin}(\widetilde{v},\widetilde{\vartheta}) \qquad \forall (\widetilde{v},\widetilde{\vartheta})\in\mathcal{A}_{\rm lin}.
\end{align}

\begin{theorem}\label{th100}{\rm [Existence result for the linear theory including terms up to order $O(h^5)$]}
	Assume that the linear operator
	$
	{\Pi}^{\rm lin}$ is bounded and
	that the following conditions concerning the initial configuration are satisfied: $\,y_0:\omega\subset\mathbb{R}^2\rightarrow\mathbb{R}^3$ is a continuous injective mapping and
	\begin{align}\label{26lin}
	{y}_0&\in{\rm H}^1(\omega ,\mathbb{R}^3),\qquad \qquad
	\nabla_x\Theta(0)\in {\rm L}^\infty(\omega ,\mathbb{R}^{3\times 3}),\qquad\qquad \det[\nabla_x\Theta(0)] \geq\, a_0 >0\,,
	\end{align}
	where $a_0$ is a constant.
	Then, for sufficiently small values of the thickness $h$ such that 
	\begin{align}\label{rcondh5int1}
	h\max\{\sup_{x\in\omega}|\kappa_1|, \sup_{x\in\omega}|\kappa_2|\}<\alpha \qquad \text{with}\qquad  \alpha<\sqrt{\frac{2}{3}(29-\sqrt{761})}\simeq 0.97083
	\end{align}
	and for constitutive coefficients  such that $\mu>0, \,\mu_{\rm c}>0$, $2\,\lambda+\mu> 0$, $b_1>0$, $b_2>0$ and $b_3>0$, the  problem \eqref{wfproblem} admits a unique solution
	$(u,\vartheta)\in  \mathcal{A}_{\rm lin}$.
\end{theorem}
\begin{proof}
	The proof is based on the Lax-Milgram theorem.  Since  the thickness $h$ satisfies \eqref{rcondh5int1}, then due to the results established in \cite{GhibaNeffPartII,GhibaNeffPartIII} for the geometrically nonlinear model, the internal energy is automatically coercive in terms of the linearised strain and curvature variables  $\mathcal{E}^{\rm{lin}}(v,\vartheta),\mathcal{K}^{\rm{lin}}(v,\vartheta)\in {\rm L^2}(\omega, \mathbb{R}^{3\times 3}) $, in the sense that for all $(v,\vartheta)\in\mathcal{A}$  there exists a constant   $a_1^+>0$  such that
	\begin{align}	\mathcal{B}^{\rm{lin}}((v,\vartheta),({v},{\vartheta}))\geq\, a_1^+\, \big( \lVert \mathcal{E}^{\rm{lin}}(v,\vartheta)\rVert ^2_{{\rm L}^2(\omega)} + \lVert \mathcal{K}^{\rm{lin}}(v,\vartheta)\rVert ^2_{{\rm L}^2(\omega)}\,\big),
	\end{align}
	where
	$a_1^+$ depends on the given constitutive coefficients.
	
	Since  the bilinear form $\mathcal{B}^{\rm lin}$ is bounded on $\mathcal{A}_{\rm lin}$, as well as   the linear operator $
	{\Pi}^{\rm lin}$, it remains to prove that the bilinear form $\mathcal{B}^{\rm lin}$ is coercive on $\mathcal{A}_{\rm lin}$. To this aim, it only remains  to prove 
	that for all $(v,\vartheta)\in\mathcal{A}_{\rm lin}$  there exists $c>0$ such that
	\begin{align}
	\lVert \mathcal{E}^{\rm{lin}}(v,\vartheta)\rVert ^2_{{\rm L}^2(\omega)} + \lVert \mathcal{K}^{\rm{lin}}(v,\vartheta)\rVert ^2_{{\rm L}^2(\omega)}\geq c\,(\lVert  v\rVert ^2_{{\rm H}^1_0(\omega)} + \lVert  \vartheta\rVert ^2_{{\rm H}^1_0(\omega)}
	).
	\end{align}
	
	The first step is to use the following alternative form of the  linearised strain measures, i.e.,
	\begin{align} \mathcal{E}^{\rm{lin}}(v,\vartheta)=  ( \nabla v - \overline{A}_\vartheta\nabla y_0\; |\; 0)\; [\nabla\Theta \,]^{-1},
	\qquad \text{
		where}\qquad 
	\overline{A}_\vartheta={\rm Anti}\, \vartheta.
	\end{align}
	The next step is to remark that, because the lifted $3\times 3$ quantity
	$
	\widehat{\rm I}_{y_0}\:\,=\,({\nabla  y_0}|n_0)^T({\nabla  y_0}|n_0)\in {\rm Sym}(3)$ is positive definite and  also it's inverse is positive definite, using also the Cauchy–Schwarz inequality and the inequality of arithmetic and geometric means we obtain
	\begin{align}\label{estE}
	\lVert \mathcal{E}^{\rm{lin}}(v,\vartheta)\rVert ^2&=\bigl\langle   ( \nabla v - \overline{A}_\vartheta\nabla y_0\; |\; 0)\; [\nabla\Theta \,]^{-1}, ( \nabla v - \overline{A}_\vartheta\nabla y_0\; |\; 0)\; [\nabla\Theta \,]^{-1}\bigr\rangle\notag\\
	&=\bigl\langle   ( \nabla v - \overline{A}_\vartheta\nabla y_0\; |\; 0)\; [\nabla\Theta \,]^{-1}[\nabla\Theta \,]^{-T}, ( \nabla v - \overline{A}_\vartheta\nabla y_0\; |\; 0)\; \bigr\rangle\notag\\&=\bigl\langle[\nabla\Theta \,]^{-1}[\nabla\Theta \,]^{-T} ( \nabla v - \overline{A}_\vartheta\nabla y_0\; |\; 0)^T, ( \nabla v - \overline{A}_\vartheta\nabla y_0\; |\; 0)^T \bigr\rangle\notag\\&=\bigl\langle\widehat{\rm I}_{y_0}^{-1} ( \nabla v - \overline{A}_\vartheta\nabla y_0\; |\; 0)^T, ( \nabla v - \overline{A}_\vartheta\nabla y_0\; |\; 0)^T \bigr\rangle\\
	&\geq \lambda_0\bigl\langle ( \nabla v - \overline{A}_\vartheta\nabla y_0\; |\; 0), ( \nabla v - \overline{A}_\vartheta\nabla y_0\; |\; 0) \bigr\rangle= \lambda_0\bigl\langle  \nabla v - \overline{A}_\vartheta\nabla y_0,  \nabla v - \overline{A}_\vartheta\nabla y_0\; \bigr\rangle\notag\\
	&\geq \lambda_0\Big[\lVert  \nabla v\rVert^2 + \lVert\overline{A}_\vartheta\nabla y_0\rVert^2-2\,\bigl\langle  \nabla v , \overline{A}_\vartheta\nabla y_0\; \bigr\rangle\Big]
	\geq \lambda_0\Big[(1-\varepsilon)\lVert  \nabla v\rVert^2 +\left(1-\frac{1}{\varepsilon}\right) \lVert\overline{A}_\vartheta\nabla y_0\rVert^2 \Big],\notag
	\end{align}
	for all $\varepsilon>0$,  where $0<\lambda_0<1$ is the smallest eigenvalue of the positive definite matrix $\widehat{\rm I}^{-1}_{y_0}$, over $\omega$.
	
	Similarly, we deduce that
	\begin{align}\label{auxine1}
	\lVert\mathcal{K}^{\rm{lin}}(v,\vartheta)\rVert^2&=\, \lVert(\nabla\vartheta\, |\, 0) \; [\nabla\Theta \,]^{-1}\rVert^2 \geq \lambda_0\, \lVert(\nabla\vartheta\, |\, 0) \rVert^2=\lambda_0\, \lVert\nabla\vartheta \rVert^2=\frac{1}{2}\lambda_0\,\sum_{\alpha=1,2} \lVert {\rm Anti} \, \partial_{x_\alpha}\vartheta  \rVert^2\\
	&=\frac{1}{2}\lambda_0\,\sum_{\alpha=1,2} \lVert \, \partial_{x_\alpha} {\rm Anti}\, \vartheta  \rVert^2=\frac{1}{2}\lambda_0\,\sum_{\alpha=1,2} \lVert \, \partial_{x_\alpha} \overline{A}_\vartheta  \rVert^2.\notag
	\end{align}
	Moreover, since $\overline{A}_\vartheta \in {\rm H}_0^1(\omega;\mathfrak{so}(3))$, we deduce the  estimate
	\begin{align}\label{auxine2}
	\sum_{\alpha=1,2} \lVert\partial_{x_\alpha} \overline{A}_\vartheta  \rVert^2_{{\rm L}^2(\omega)} &\geq c_p \lVert  \overline{A}_\vartheta  \rVert^2_{{\rm L}^2(\omega)},
	\end{align}
	where $c_p>0$ is the constant depending only on $\omega$ coming from the Poincar\'{e} inequality.  Because the Frobenius norm is sub-multiplicative, we have
\begin{align}
	\lVert  \overline{A}_\vartheta  \nabla y_0\rVert^2\leq \lVert  \overline{A}_\vartheta \rVert^2\,\lVert \nabla y_0\rVert^2\leq \lVert  \overline{A}_\vartheta \rVert^2	\sup_{(x_1,x_2)\in {\omega}} \,\lVert \nabla y_0\rVert^2
	\end{align}
	and therefore
	\begin{align}\label{auxin}	
	\lVert  \overline{A}_\vartheta  \nabla y_0\rVert^2_{{\rm L}^2(\omega)}\leq \lVert  \overline{A}_\vartheta \rVert^2_{{\rm L}^2(\omega)}\,\sup_{(x_1,x_2)\in {\omega}} \,\lVert \nabla y_0\rVert^2.
	\end{align}\newpage
Using \eqref{auxine1}, \eqref{auxine2} and \eqref{auxin} 	we deduce
	\begin{align}\label{estK}
	\lVert\mathcal{K}^{\rm{lin}}(v,\vartheta)\rVert^2_{{\rm L}^2(\omega)}&=\frac{1}{2}\,\lVert\mathcal{K}^{\rm{lin}}(v,\vartheta)\rVert^2_{{\rm L}^2(\omega)}+\frac{1}{2}\,\lVert\mathcal{K}^{\rm{lin}}(v,\vartheta)\rVert^2_{{\rm L}^2(\omega)}\notag\\& \geq\frac{1}{4}\lambda_0\,\sum_{\alpha=1,2} \lVert \, \partial_{x_\alpha} \overline{A}_\vartheta  \rVert^2+\,\frac{c_p\,\lambda_0}{4} \lVert  \overline{A}_\vartheta  \rVert^2_{{\rm L}^2(\omega)}\notag\\& \geq\frac{1}{4}\lambda_0\,\sum_{\alpha=1,2} \lVert \, \partial_{x_\alpha} \overline{A}_\vartheta  \rVert^2+\,\frac{c_p\,\lambda_0}{4} \frac{1}{\sup_{(x_1,x_2)\in {\omega}}\lVert   \nabla y_0  \rVert^2}\lVert  \overline{A}_\vartheta  \nabla y_0\rVert^2_{{\rm L}^2(\omega)} \\&=\frac{1}{4}\lambda_0\,\sum_{\alpha=1,2} \lVert \, \partial_{x_\alpha} \overline{A}_\vartheta  \rVert^2+\frac{c_p\,\lambda_0}{4\,\sup_{(x_1,x_2)\in {\omega}}\tr({\rm I}_{y_0})}\lVert  \overline{A}_\vartheta  \nabla y_0\rVert^2_{{\rm L}^2(\omega)}\notag \\&\geq\frac{1}{4}\lambda_0\,\sum_{\alpha=1,2} \lVert \, \partial_{x_\alpha} \overline{A}_\vartheta  \rVert^2+ \frac{c_p\,\lambda_0}{8\,\lambda_M}\lVert  \overline{A}_\vartheta  \nabla y_0\rVert^2_{{\rm L}^2(\omega)},\notag
	\end{align}
	where $\lambda_M>1$ is the largest eigenvalue of the given  positive definite matrix $\widehat{\rm I}_{y_0}$, over $\omega$.
	Hence with  \eqref{estK} we have
	\interdisplaylinepenalty=10000
	\begin{align}
	\lVert \mathcal{E}^{\rm{lin}}(v,\vartheta)\rVert ^2_{{\rm L}^2(\omega)} + \lVert \mathcal{K}^{\rm{lin}}(v,\vartheta)\rVert ^2_{{\rm L}^2(\omega)} 
	&\geq \lambda_0\Big[(1-\varepsilon)\lVert  \nabla v\rVert^2_{{\rm L}^2(\omega)} +\left(1-\frac{1}{\varepsilon} +\frac{c_p}{8\,\lambda_M}\right)\lVert  \overline{A}_\vartheta  \nabla y_0\rVert^2_{{\rm L}^2(\omega)}\\&\quad \qquad \quad +\frac{1}{2}\,\sum_{\alpha=1,2} \lVert \, \partial_{x_\alpha} \overline{A}_\vartheta  \rVert^2 _{{\rm L}^2(\omega)}\Big]\notag
	\end{align}
	for all $\varepsilon>0$. Now we are looking for some $\varepsilon>0$ such that in parallel 
	\begin{align}
	1-\varepsilon>0 \qquad \text{and}\qquad 1-\frac{1}{\varepsilon} +\frac{c_p}{8\,\tr({\rm I}_{y_0})}>0, \qquad 
	\text{i.e.},
\qquad 1>\varepsilon>\frac{1}{1 +\frac{c_p}{8\,\tr({\rm I}_{y_0})}}.
	\end{align}
	Since, $c_p>0$ and $\lambda_M>0$, such a constant $\varepsilon>0$ exists and there is a positive constant $c_1>0$ such that the following Korn-type inequality, see  \cite{Birsan08} for some related Korn-type inequalities in the Cosserat theory  with a deformable vector, is satisfied
	\begin{align}
	\lVert \mathcal{E}^{\rm{lin}}(v,\vartheta)\rVert ^2_{{\rm L}^2(\omega)}  + \lVert \mathcal{K}^{\rm{lin}}(v,\vartheta)\rVert ^2_{{\rm L}^2(\omega)} 
	&\geq c_1\Big[\lVert  \nabla v\rVert^2_{{\rm L}^2(\omega)}  +\lVert  \overline{A}_\vartheta  \nabla y_0\rVert^2_{{\rm L}^2(\omega)} +\lVert  \nabla \overline{A}_\vartheta  \rVert^2_{{\rm L}^2(\omega)}\Big]\notag\\
	&\geq c_1\Big[\lVert  \nabla v\rVert^2_{{\rm L}^2(\omega)}  +\lVert \nabla  {\rm Anti}\,  \vartheta  \rVert^2_{{\rm L}^2(\omega)}\Big],
	\end{align}
	i.e., a positive constant $c>0$ such that 
	\begin{align}
	\lVert \mathcal{E}^{\rm{lin}}(v,\vartheta)\rVert ^2_{{\rm L}^2(\omega)}  + \lVert \mathcal{K}^{\rm{lin}}(v,\vartheta)\rVert ^2_{{\rm L}^2(\omega)} 
	&\geq c\Big[\lVert  \nabla v\rVert^2_{{\rm L}^2(\omega)}  +\lVert \nabla \vartheta  \rVert^2_{{\rm L}^2(\omega)}\Big]\qquad \forall\  (v,\vartheta)\in \mathcal{A}^{\rm lin}.
	\end{align}
	Hence, the bilinear form $\mathcal{B}^{\rm lin}$ is coercive and the Lax-Milgram theorem leads to the conclusion of the theorem.
\end{proof}

\subsubsection{Existence result for the linearised truncated $O(h^3)$- Cosserat shell model}

As a restricted case of the linear Cosserat $O(h^5)$-shell model  we obtain the linear Cosserat $O(h^3)$-shell model, by ignoring the terms of order $O(h^5)$, i.e. the variational problem is  to find a midsurface displacement vector field 
$v:\omega\subset\mathbb{R}^2\to\mathbb{R}^3$ and the micro-rotation vector field $\vartheta:\omega\subset\mathbb{R}^2\to\mathbb{R}^3$ minimizing on $\omega$:
\begin{align}\label{e89lh3}
I(v,\vartheta)\!=\!\! \int_{\omega}   \!\!\Big[  &\Big(h+{\rm K}\,\dfrac{h^3}{12}\Big)\,
W_{\mathrm{shell}}\big(    \mathcal{E}_{m,s}^{\rm{lin}}  \big)+  \dfrac{h^3}{12}\,
W_{\mathrm{shell}}  \big(   \mathcal{E}_{m,s}^{\rm{lin}}  \, {\rm B}_{y_0} +   {\rm C}_{y_0} \mathcal{K}_{e,s}^{\rm{lin}}  \big)  \notag \\&
-\dfrac{h^3}{3} \mathrm{ H}\,\mathcal{W}_{\mathrm{shell}}  \big(  \mathcal{E}_{m,s}^{\rm{lin}}  ,
\mathcal{E}_{m,s}^{\rm{lin}} {\rm B}_{y_0}+{\rm C}_{y_0}\, \mathcal{K}_{e,s}^{\rm{lin}}  \big)+
\dfrac{h^3}{6}\, \mathcal{W}_{\mathrm{shell}}  \big(  \mathcal{E}_{m,s}^{\rm{lin}}  ,
( \mathcal{E}_{m,s}^{\rm{lin}} {\rm B}_{y_0}+{\rm C}_{y_0}\, \mathcal{K}_{e,s}^{\rm{lin}} ){\rm B}_{y_0} \big) \\&+ \,\Big(h-{\rm K}\,\dfrac{h^3}{12}\Big)\,
W_{\mathrm{curv}}\big(  \mathcal{K}_{e,s}^{\rm{lin}}  \big)    +  \dfrac{h^3}{12}\,
W_{\mathrm{curv}}\big(  \mathcal{K}_{e,s} ^{\rm{lin}}   {\rm B}_{y_0} \,  \big) 
\Big] {\rm det}(\nabla y_0|n_0)     d a- \overline{\Pi}^{\rm lin}(v,\vartheta)\,,\notag
\end{align}
where
$\overline{\Pi}^{\rm lin}(v,\vartheta)$ is the linearization of $\overline{\Pi}(m,\overline{Q}_{e,s})$. The weak form  of the equilibrium problem of the linear theory of Cosserat shells including terms up to order $O(h^3)$ is to find
$(v,\vartheta)\in\mathcal{A}_{\rm lin}$ satisfying 
\begin{align}\label{wfproblemh3}
\mathcal{B}^{\rm{lin}}_{h^3 }((v,\vartheta),(\widetilde{v},\widetilde{\vartheta}))={\Pi}^{\rm lin}(\widetilde{v},\widetilde{\vartheta}) \qquad \forall\, (\widetilde{v},\widetilde{\vartheta})\in\mathcal{A}_{\rm lin},
\end{align}
where $\mathcal{B}_{h^3 }^{\rm{lin}}:\mathcal{A}_{\rm lin}\times \mathcal{A}_{\rm lin}\to \mathbb{R},$
\begin{align}
\mathcal{B}_{h^3 }^{\rm{lin}}((v,\vartheta),(\widetilde{v},\widetilde{\vartheta})):=\int_{\omega}   \!\!\Big[  &\Big(h+{\rm K}\,\dfrac{h^3}{12}\Big)\,
\mathcal{W}_{\mathrm{shell}}\big(    \mathcal{E}^{\rm{lin}}(v,\vartheta), \mathcal{E}^{\rm{lin}}(\widetilde{v},\widetilde{\vartheta})  \big)\notag\\&+  \dfrac{h^3}{12}\,
\mathcal{W}_{\mathrm{shell}}  \big(   \mathcal{E}^{\rm{lin}}(v,\vartheta)  \, {\rm B}_{y_0} +   {\rm C}_{y_0} \mathcal{K}^{\rm{lin}}(v,\vartheta), \mathcal{E}^{\rm{lin}}(\widetilde{v},\widetilde{\vartheta})  \, {\rm B}_{y_0} +   {\rm C}_{y_0} \mathcal{K}^{\rm{lin}}(\widetilde{v},\widetilde{\vartheta})  \big)  \notag\\&
-\dfrac{h^3}{6} \mathrm{ H}\,\mathcal{W}_{\mathrm{shell}}  \big(  \mathcal{E}^{\rm{lin}}(v,\vartheta)  ,
\mathcal{E}^{\rm{lin}}(\widetilde{v},\widetilde{\vartheta}) {\rm B}_{y_0}+{\rm C}_{y_0}\, \mathcal{K}^{\rm{lin}}(\widetilde{v},\widetilde{\vartheta})  \big)\notag \\
&-\dfrac{h^3}{6} \mathrm{ H}\,\mathcal{W}_{\mathrm{shell}}  \big(  \mathcal{E}^{\rm{lin}}(\widetilde{v},\widetilde{\vartheta})  ,
\mathcal{E}^{\rm{lin}}(v,\vartheta) {\rm B}_{y_0}+{\rm C}_{y_0}\, \mathcal{K}^{\rm{lin}}(v,\vartheta) \big)\notag\\&+
\dfrac{h^3}{12}\, \mathcal{W}_{\mathrm{shell}}  \big(  \mathcal{E}^{\rm{lin}}(v,\vartheta)  ,
( \mathcal{E}^{\rm{lin}}(\widetilde{v},\widetilde{\vartheta}) {\rm B}_{y_0}+{\rm C}_{y_0}\, \mathcal{K}^{\rm{lin}}(\widetilde{v},\widetilde{\vartheta}) ){\rm B}_{y_0} \big)
\\&+
\dfrac{h^3}{12}\, \mathcal{W}_{\mathrm{shell}}  \big(  \mathcal{E}^{\rm{lin}}(\widetilde{v},\widetilde{\vartheta})  ,
( \mathcal{E}^{\rm{lin}}(v,\vartheta) {\rm B}_{y_0}+{\rm C}_{y_0}\, \mathcal{K}^{\rm{lin}}(v,\vartheta) ){\rm B}_{y_0} \big)\notag \\&+ \,\Big(h-{\rm K}\,\dfrac{h^3}{12}\Big)\,
\mathcal{W}_{\mathrm{curv}}\big(  \mathcal{K}^{\rm{lin}}(v,\vartheta), \mathcal{K}^{\rm{lin}}(\widetilde{v},\widetilde{\vartheta})  \big)  \notag \\& +  \dfrac{h^3}{12}\,
\mathcal{W}_{\mathrm{curv}}\big(  \mathcal{K}^{\rm{lin}}(v,\vartheta)   {\rm B}_{y_0},\mathcal{K}^{\rm{lin}}(\widetilde{v},\widetilde{\vartheta})    {\rm B}_{y_0} \,  \big) 
\Big] {\rm det}(\nabla y_0|n_0)     d a\notag.
\end{align}
and the linear operator
\begin{align}
{\Pi}^{\rm lin}:\mathcal{A}_{\rm lin}\to\mathbb{R}, \qquad\qquad\qquad  {\Pi}^{\rm lin}(\widetilde{v},\widetilde{\vartheta})=\overline{\Pi}^{\rm lin}(\widetilde{v},\widetilde{\vartheta}).
\end{align}

A similar proof as in the case of the $O(h^5)$ model leads us to the following results (see also the proof of the coercivity from \cite{GhibaNeffPartII,GhibaNeffPartIII})
\begin{theorem}\label{th1h3}{\rm [Existence result for the truncated linear theory including terms up to order $O(h^3)$]}
	Assume that the linear operator
	$
	{\Pi}^{\rm lin}$ is bounded and
	that the following conditions concerning the initial configuration are satisfied: $\,y_0:\omega\subset\mathbb{R}^2\rightarrow\mathbb{R}^3$ is a continuous injective mapping and
	\begin{align}\label{26lin2}
	{y}_0&\in{\rm H}^1(\omega ,\mathbb{R}^3),\qquad\qquad\qquad
	\nabla_x\Theta(0)\in {\rm L}^\infty(\omega ,\mathbb{R}^{3\times 3}),\qquad\qquad\qquad \det[\nabla_x\Theta(0)] \geq\, a_0 >0\,,
	\end{align}
	where $a_0$ is a constant.
	Then, if the thickness $h$ satisfies either of the following conditions:
	\begin{enumerate}
		\item[i)] $	h\max\{\sup_{x\in\omega}|\kappa_1|, \sup_{x\in\omega}|\kappa_2|\}<\alpha$ \quad {\bf and} \quad  $	h^2<\frac{(5-2\sqrt{6})(\alpha^2-12)^2}{4\, \alpha^2}\frac{ {c_2^+}}{C_1^+}$  \quad  \text{with}  \quad  $\quad 0<\alpha<2\sqrt{3}$;
		\item[ii)] $h\max\{\sup_{x\in\omega}|\kappa_1|, \sup_{x\in\omega}|\kappa_2|\}<\frac{1}{a}$  \quad {\bf and} \quad  $a>\max\Big\{1 + \frac{\sqrt{2}}{2},\frac{1+\sqrt{1+3\frac{C_1^+}{c_1^+}}}{2}\Big\}$,
	\end{enumerate}
	where  $c_2^+$  denotes the smallest eigenvalue  of
	$
	W_{\mathrm{curv}}(  S ),
	$
	and $c_1^+$ and $ C_1^+>0$ denote the smallest and the biggest eigenvalues of the quadratic form $W_{\mathrm{shell}}^{\infty}(  S)$,
	and for constitutive coefficients  such that $\mu>0$, $2\,\lambda+\mu> 0$, $b_1>0$, $b_2>0$, $b_3>0$ and $L_{\rm c}>0$, the  problem \eqref{wfproblemh3} admits a unique solution
	$(u,\vartheta)\in  \mathcal{A}_{\rm lin}$.
\end{theorem}
\begin{proof}
	The proof is based on the coercivity inequality already proven in \cite{GhibaNeffPartII,GhibaNeffPartIII}, see Proposition \ref{coerh3r}, and the same steps as in the proof of Theorem \ref{th100}.
\end{proof}

\begin{remark}
	We observe that  the conditions imposed on the thickness are more restrictive in the truncated $O(h^3)$ model than in the $O(h^5)$ model. In other words, the existence results hold true in $O(h^3)$ only for smaller values of the thickness $h$, in comparison to  the $O(h^5)$ model. Moreover, while in the $O(h^5)$ model the conditions imposed on the thickness are independent on the constitutive coefficients (the same  conditions for all shells, i.e., all  materials), in the $O(h^3)$ model the conditions  depend on the assumed  constitutive coefficients.
\end{remark}

\section{The linear model for the Cosserat flat shell: a consistency check}\setcounter{equation}{0}

For the Cosserat flat shell model (flat initial configuration) obtained in \cite{Neff_plate04_cmt} we have $\Theta(x_1,x_2,x_3)=(x_1,x_2,x_3)$ which gives $\nabla_x\Theta=\id_3$ and $y_0(x_1,x_2)=(x_1,x_2):={\text{id}}(x_1,x_2)$. Also $Q_0=\id_3$, $n_0=e_3$, $\overline{Q}_{e,0}(x_1,x_2)=\overline{R}(x_1,x_2)$, ${\rm B}_{\rm id}\,=\,0_3,$ ${\rm C}_{\rm id}\,=\,\begin{footnotesize}\begin{pmatrix}
			0&1&0 \\
			-1&0&0 \\
			0&0&0
\end{pmatrix}\end{footnotesize}\in \mathfrak{so}(3),$ $
{\rm L}_{\rm id}=0_2,$ ${\rm K}=0,$ and $
{\rm H}=\,0$.
Hence,  see \eqref{ec11},   $\mathcal{E}_{m,s}^{\rm{lin}}$  and  $\mathcal{K}_{e,s}^{\rm{lin}}$ are
\begin{align}
\mathcal{E}_{m,s}^{\rm{lin,plate}} :&= ( \nabla v |\; 0)-(\overline{A}_\vartheta.\,e_1\;|\;\overline{A}_\vartheta.e_2\; |\; 0)=( \nabla v |\; \overline{A}_\vartheta.e_3)-\overline{A}_\vartheta\\&=( \nabla v\,|\,0) -  \vartheta\times(\,e_1\,|\,e_2\,|\,0)=( \partial_{x_1} v- \vartheta\times e_1\;|\;\partial_{x_2} v- \vartheta\times e_2 \;|\;0)  \notag
\end{align}
and
\begin{align}
\mathcal{K}_{e,s}^{\rm{lin,plate}} &:= \Big(\mbox{axl}(\partial_{x_1}\overline{A}_\vartheta)\,|\, \mbox{axl}(\partial_{x_2}\overline{A}_\vartheta)\,|\, 0 \Big) \;  = \, (\nabla\vartheta\, |\, 0) \; ,
\end{align}
respectively.
Moreover, due to \eqref{ec11}, we deduce
\begin{align}
\mathrm{C}_{y_0}\,\mathcal{K}_{e,s}^{\rm{lin}} &= - e_3\times  (\nabla\vartheta\, |\, 0)\; .
\end{align}

The  variational problem for the linear Cosserat plate model  is  to find a midsurface displacement vector field 
$v:\omega\subset\mathbb{R}^2\to\mathbb{R}^3$ and the micro-rotation vector field $\vartheta:\omega\subset\mathbb{R}^2\to\mathbb{R}^3$ minimizing on $\omega$:
\begin{align}\label{e89lplate}
I(v,\vartheta)\!=\!\! \int_{\omega}   \!\!
\Big[ h\,W_{\mathrm{shell}}\big(    \mathcal{E}_{m,s}^{\rm{lin,plate}}  \big)
+  \dfrac{h^3}{12}\, W_{\mathrm{shell}}  \big(   {\rm C}_{y_0} \mathcal{K}_{e,s}^{\rm{lin,plate}}  \big) 
 + \,h\,W_{\mathrm{curv}}\big(  \mathcal{K}_{e,s}^{\rm{lin,plate}}  \big)   
\Big] d a- {\Pi}^{\rm lin,plate}(v,\vartheta)\,,
\end{align}
where $\mathcal{E}_{m,s}^{\rm{lin,plate}}$, $\mathcal{K}_{e,s}^{\rm{lin,plate}}$ and $\mathrm{C}_{\rm id}\,\mathcal{K}_{e,s}^{\rm{lin}}$ have the expressions from above and  
$\overline{\Pi}^{\rm lin,plate}(v,\vartheta)$ is the linearization of $\overline{\Pi}(m,\overline{Q}_{e,s})$ for a flat initial configuration. Note the automatic absence of $O(h^5)$-terms for the flat shell problem and
that the partial derivatives of the third component of $\vartheta$  survive after linearization only in the curvature part of the energy. 

Let us consider the following alternative expression of $\mathcal{E}_{m,s}^{\rm{lin,plate}}$ and $\mathrm{C}_{\rm id} \mathcal{K}_{e,s}^{\rm{lin,plate}}$  
\begin{align}\label{eq5plate}
\mathcal{E}_{m,s}^{\rm{lin,plate}}=&
\begin{footnotesize}\left( \begin{array}{c|c}
\mathcal{G}^{\rm{lin,plate}} & 0 \vspace{4pt}\\
\mathcal{T}^{\rm{lin,plate}}  & 0
\end{array} \right)\end{footnotesize} ,
\qquad\qquad \qquad
\mathrm{C}_{\rm id} \mathcal{K}_{e,s}^{\rm{lin,plate}}= \, -
\begin{footnotesize}\left( \begin{array}{c|c}
\mathcal{R}^{\rm{lin,plate}} & 0 \vspace{4pt}\\
0 & 0
\end{array} \right)\end{footnotesize}
,
\end{align}
where
\begin{align}
\mathcal{G}^{\rm{lin,plate}} &=  (e_1\,|\,e_2)^{T} ( \nabla v - \overline{A}_\vartheta(e_1\,|\,e_2) )=    \begin{footnotesize}
\begin{pmatrix}
\partial_{x_1}v_1&\partial_{x_2}v_1+\vartheta_3\\
\partial_{x_1}v_2-\vartheta_3&\partial_{x_2}v_2
\end{pmatrix}\end{footnotesize},\notag
\\
\mathcal{T}^{\rm{lin,plate}} &=  e_3^T ( \nabla v - \overline{A}_\vartheta (e_1\,|\,e_2))= \begin{footnotesize}
\begin{pmatrix}
\partial_{x_1}v_3+\vartheta_2&\partial_{x_2}v_3-\vartheta_1
\end{pmatrix}\end{footnotesize}\\
\mathcal{R}^{\rm{lin,plate}} &=   (e_1\,|\,e_2)^{T} ( e_3\times\nabla \vartheta )=[(e_1\,|\,e_2)^{T} \times e_3^T](\nabla \vartheta )=-[e_3\times (e_1\,|\,e_2)]^T(\nabla \vartheta )\notag\\&=- (e_3\times e_1\,|\,e_3\times e_2)]^T(\nabla \vartheta )=- (e_2\,|- e_1)^T\,\nabla \vartheta=\,\begin{footnotesize}\begin{pmatrix}
0&-1&0 \\
1&0&0 \\
\end{pmatrix}\end{footnotesize}\,\begin{pmatrix}
\partial_{x_1}\vartheta_1&\partial_{x_2}\vartheta_1\\
\partial_{x_1}\vartheta_2&\partial_{x_2}\vartheta_2 \\
\partial_{x_1}\vartheta_3&\partial_{x_2}\vartheta_3
\end{pmatrix}=\,\begin{footnotesize}\begin{pmatrix}
-\partial_{x_1}\vartheta_2&-\partial_{x_2}\vartheta_2\\
\partial_{x_1}\vartheta_1&\partial_{x_2}\vartheta_1  \\
\end{pmatrix}\end{footnotesize} ,\notag
\end{align}
as well as the decomposition 
\begin{align}  \label{descK0}
\mathcal{K}_{e,s}^{\rm{lin,plate}} =& \, \mathrm{C}_{\rm id}( - \mathrm{C}_{\rm id} \mathcal{K}_{e,s}^{\rm{lin,plate}}) +(0|0|e_3) \,(0|0|(\mathcal{K}_{e,s}^{\rm{lin,plate}})^T\,e_3)^T\, \, \\
=& \,\mathrm{C}_{\rm id}
\begin{footnotesize}\left( \begin{array}{c|c}
\mathcal{R}^{\rm{lin,plate}} & 0 \vspace{4pt}\\
0 & 0
\end{array} \right)\end{footnotesize} +(0|0|e_3) \,(0|0|(\mathcal{N}^{\rm{lin,plate}})^T)^T,\notag
\end{align} 
where 
\begin{equation}
\label{e5d0}
\mathcal{N}^{\rm{lin,plate}} :=  e_3^T\nabla \vartheta =(\partial_{x_1}\vartheta_3,\partial_{x_2}\vartheta_3),
\end{equation}
represents the row vector of {\it drilling bendings}. Then 
the  variational problem for the linear Cosserat plate model  is  to find a midsurface displacement vector field 
$v:\omega\subset\mathbb{R}^2\to\mathbb{R}^3$ and the micro-rotation vector field $\vartheta:\omega\subset\mathbb{R}^2\to\mathbb{R}^3$ minimizing on $\omega$:\newpage
\begin{align}\label{e89lplate}
I(v,\vartheta)\!=\!\! \int_{\omega}   \!\!
\Big[  &\underbrace{h\, \Big( \mu\,\lVert \mathrm{sym}\,     \mathcal{G}^{\rm{lin,plate}}\rVert^2 +    \mu_c\lVert \,\mathrm{skew}\,  \mathcal{G}^{\rm{lin,plate}}\rVert^2  + \,\dfrac{\lambda\,\mu}{\lambda+2\mu}\,[\mathrm{tr}  \big(  \mathcal{G}^{\rm{lin,plate}}\big)]^2\Big)}_{\textrm{in-plane deformation}}\notag\\&
+ \underbrace{h\,\dfrac{\mu+\mu_c}{2}
	\bigl\lVert \mathcal{T}^{\rm{lin,plate}} \rVert^2}_{\textrm{transverse shear}}  \\\notag&+ \underbrace{\dfrac{h^3}{12}\, \Big( \mu\,\lVert \mathrm{sym}\,    \mathcal{R}^{\rm{lin,plate}}\rVert^2 +    \mu_c\lVert \,\mathrm{skew}\,   \mathcal{R}^{\rm{lin,plate}}\rVert^2  + \,\dfrac{\lambda\,\mu}{\lambda+2\mu}\,[\mathrm{tr}  \big( \mathcal{R}^{\rm{lin,plate}}\big)]^2\Big)}_{{\rm bending-type\  terms of order} \ \ O(h^3)}  \\
&+  \underbrace{h\,\mu \,L_c^2\Big[ b_1\,\lVert \mathrm{sym} \,\mathcal{R}^{\rm{lin,plate}} \rVert^2 
	+     \Big(8\,b_3+\dfrac{b_1}{3}\Big)\;\lVert \,\mathrm{skew}\mathcal{R}^{\rm{lin,plate}}  \rVert^2 + \,\dfrac{b_2-b_1}{2}\,\big[\mathrm{tr}  \big( \mathcal{R}^{\rm{lin,plate}} \big)\,\big]^2\,\Big]}_{{\rm bending-type terms of order} \ \ O(h)}\notag
\vspace{6pt}\\
&\ \ + \underbrace{h\,\mu\, L_c^2\,\dfrac{b_1+b_2}{2}\,
 \lVert  \mathcal{N}^{\rm{lin,plate}}\rVert^2}_{\textrm{drilling bendings}}  
\Big] d a- {\Pi}^{\rm lin,plate}(v,\vartheta)\,,
\end{align}

Since $
\sym(\mathcal{G}^{\rm{lin}} )= \mathcal{G}_{\rm{Koiter}}^{\rm{lin}},$ we observe that the in-plane deformation energy coincides with the membrane part of the energy of the Koiter model. In addition, in our linear model the influence of the microstructure is  taken into account (see the terms having the coefficient $\mu_{\rm c}$) and  traverse shear is taken into account, too. The connection between the bending strain tensor in the Koiter model and the bending strain tensor in our linear model may be done in a constrained model, since the Koiter model does not consider Cosserat effects. In the flat shell model the bending-type terms appear due to both three dimensional energies, the elastic strain energy and the curvature energy. The two-dimensional part resulting after dimensional reduction from the curvature energy is capable to capture the drilling bendings, too.

\section{A comparison  with  the general 6-parameter shell linear model}

There exist many linear shell model derived in the framework of Cosserat theory. Most of them are included as particular cases in the linearised model of the 6-parameter theory of shells \cite{Pietraszkiewicz-book04,Pietraszkiewicz09,Pietraszkiewicz10,Pietraszkiewicz14,Eremeyev06}.	In the resultant 6-parameter theory of shells, the strain energy density for isotropic shells has been presented in various forms. The comparison of our nonlinear shell model and the  general 6-parameter nonlinear shell  models was done in \cite{GhibaNeffPartI}. Since the models are quadratic in all the involved strain and curvature tensors, all the remarks from this comparison are still valid in the linear models. We recall here the concluding remarks of the comparison, but in terms of the strain of the linear model. 
	
	The simplest expression $W_{\rm P}(\mathcal{E}_{m,s}^{\rm lin} ,\mathcal{K}_{e,s}^{\rm lin})$ has been proposed in the papers \cite{Pietraszkiewicz-book04,Pietraszkiewicz10} in the form
\begin{align}\label{56}
2\,W_{\rm P}(\mathcal{E}_{m,s}^{\rm lin} ,\mathcal{K}_{e,s}^{\rm lin})=& \,C\big[\,\nu \,(\mathrm{tr}\, \mathcal{E}_{m,s}^{{\rm lin},\parallel})^2 +(1-\nu)\, \mathrm{tr}((\mathcal{E}_{m,s}^{{\rm lin},\parallel} )^T \mathcal{E}_{m,s}^{{\rm lin},\parallel} )\big]  + \alpha_{s\,}C(1-\nu) \, \lVert  \mathcal{E}_{m,s}^{{\rm lin},T}  {n}_0\rVert^2\notag \\
&+\,D \big[\,\nu\,(\mathrm{tr}\, \mathcal{K}_{e,s}^{{\rm lin},\parallel})^2 + (1-\nu)\, \mathrm{tr}((\mathcal{K}_{e,s}^{{\rm lin},\parallel} )^T \mathcal{K}_{e,s}^{{\rm lin},\parallel} )\big]  + \alpha_{t\,}D(1-\nu) \,     \lVert \mathcal{K}_{e,s}^{{\rm lin},T}  {n}_0\rVert^2,
\end{align}
where the decompositions  of $\mathcal{E}_{m,s}^{\rm lin}$ and $\mathcal{K}_{e,s}^{\rm lin}$ into two orthogonal directions (in the tangential plane and in the normal direction)\footnote{	Here, we have used that, since
	${\rm A}_{y_0}\,=\,\id_3-n_0\otimes n_0$, for all $X\in \mathbb{R}^{3\times 3}$ it holds
$
	\lVert X^{\perp}\rVert^2=\lVert X^T\,n_0\rVert^2.
$}  are considered, i.e.,
\begin{align}\label{54,1}
\mathcal{E}_{m,s}^{{\rm lin},\parallel}&={\rm A}_{y_0}\mathcal{E}_{m,s}^{\rm lin}=(\id_3-n_0\otimes n_0)\mathcal{E}_{m,s}^{\rm lin},\qquad\qquad\quad\ \mathcal{K}_{e,s}^{{\rm lin},\parallel}={\rm A}_{y_0}\,\mathcal{K}_{e,s}^{\rm lin}=(\id_3-n_0\otimes n_0) \,\mathcal{K}_{e,s}^{\rm lin},\\
\mathcal{E}_{m,s}^{{\rm lin},\perp}&=(\id_3-{\rm A}_{y_0})\,\mathcal{E}_{m,s}^{\rm lin}=n_0\otimes n_0\,\mathcal{E}_{m,s}^{\rm lin},\qquad\qquad\quad \mathcal{K}_{e,s}^{{\rm lin},\perp}=(\id_3-{\rm A}_{y_0})\,\mathcal{K}_{e,s}^{\rm lin}=n_0\otimes n_0\,\mathcal{K}_{e,s}^{\rm lin}.\notag
\end{align}

The constitutive coefficient $C=\frac{E\,h}{1-\nu^2}\,$ is the stretching (in-plane) stiffness of the shell, $D=\frac{E\,h^3}{12(1-\nu^2)}\,$ is the bending stiffness, and $\alpha_s\,$, $\alpha_t$ are two shear correction factors. Also,    $ E=\frac{\mu\,(3\,\lambda+2\,\mu)}{\lambda+\mu}$ and $ \nu=\frac{\lambda}{2\,(\lambda+\mu)}
$ denote the Young modulus and  Poisson ratio of the isotropic and homogeneous material. In the numerical treatment of non-linear shell problems, the values of the shear correction factors have been set to  $\alpha_s=5/6$, $\alpha_t=7/10$ in \cite{Pietraszkiewicz10}. The value $\alpha_s=5/6$ is a classical suggestion, which has been previously deduced analytically by Reissner in the case of plates \cite{Reissner45,Naghdi72}.
Also, the value $\,\alpha_t=7/10\,$ was proposed earlier in \cite[see p.78]{Pietraszkiewicz-book79} and has been suggested in the work \cite{Pietraszkiewicz79}.
However, the discussion concerning the possible values of  shear correction factors for shells is long and controversial in the literature \cite{Naghdi72,Naghdi-Rubin95}.

	The coefficients in \eqref{56} are expressed in terms of the Lam\'e constants of the material $\lambda$ and $\mu$ now  by the relations
\begin{equation*}
C \,\nu\,=\frac{4\,\mu\,(\lambda+\mu)}{\lambda+2\,\mu}\,h\,,\qquad\quad  C(1\!-\!\nu)=2\,\mu\, h,\qquad\quad
D \,\nu\,=\frac{4\,\mu\,(\lambda+\mu)}{\lambda+2\,\mu}\,\frac{h^3}{12}\,,\qquad\quad 
D(1\!-\!\nu)=\mu\,\dfrac{ h^3}{6}\,.
\end{equation*}

	In \cite{Eremeyev06}, Eremeyev and Pietraszkiewicz have proposed a more general form of the strain energy density, namely
\begin{align}\label{58}
2\, W_{\rm EP}(\mathcal{E}_{m,s}^{\rm lin} ,\mathcal{K}_{e,s}^{\rm lin})=& \, \alpha_1\,\big(\mathrm{tr} \, \mathcal{E}_{m,s}^{{\rm lin},\parallel}\big)^2 +\alpha_2 \, \mathrm{tr} \big(\mathcal{E}_{m,s}^{{\rm lin},\parallel}\big)^2    + \alpha_3 \,\mathrm{tr}\big((\mathcal{E}_{m,s}^{{\rm lin},\parallel})^T \mathcal{E}_{m,s}^{{\rm lin},\parallel} \big)  + \alpha_4  \,   \lVert  \mathcal{E}_{m,s}^{{\rm lin},T}  {n}_0\rVert^2 \notag\\
& + \beta_1\,\big(\mathrm{tr}\,  \mathcal{K}_{e,s}^{{\rm lin},\parallel}\big)^2 +\beta_2\,  \mathrm{tr} \big(\mathcal{K}_{e,s}^{{\rm lin},\parallel}\big)^2    + \beta_3\, \mathrm{tr}\big((\mathcal{K}_{e,s}^{{\rm lin},\parallel})^T \mathcal{K}_{e,s}^{{\rm lin},\parallel} \big)  + \beta_4 \,    \lVert  \mathcal{K}_{e,s}^{{\rm lin},T}  {n}_0\rVert^2.
\end{align}
Already, note the absence of coupling terms involving $\mathcal{K}_{e,s}^{{\rm lin},\parallel}$ and $\mathcal{E}_{m,s}^{{\rm lin},\parallel}$. 

The eight coefficients $\alpha_k$, $\beta_k$ ($k=1,2,3,4$) can depend in general on the structure curvature tensor  \break $\mathcal{K}^0=Q_0 ( \, \text{axl}(Q_{0}^T\partial_{x_1} Q_{0})  \, | \,  \text{axl}(Q_{0}^T\partial_{x_2} Q_{0}) \, |\, \,0 \, )(\nabla y_0\,|\,n_0)^{-1}$ of the reference configuration, where $Q_0={\rm polar}(\nabla y_0\,|\,n_0)$. 

We conclude
\begin{remark}
	\begin{itemize}\item[]
		\item[i)]  By comparing our $W_{\rm our}\big(  \mathcal{E}_{m,s}^{{\rm lin}},\mathcal{K}_{e,s}^{{\rm lin}}  \big)$ with $W_{\rm EP}\big(  \mathcal{E}_{m,s}^{{\rm lin}},\mathcal{K}_{e,s}^{{\rm lin}}  \big)$ we deduce the following identification of the constitutive coefficients $\alpha_1\,,...,\alpha_4,\beta_1\,,...,\beta_4$
		\begin{align}
		\alpha_1&=\Big(h+{\rm K}\,\dfrac{h^3}{12}\Big)\,\dfrac{2\mu\lambda}{2\mu+\lambda}\,,\qquad \qquad\alpha_2=\Big(h+{\rm K}\,\dfrac{h^3}{12}\Big)\,(\mu-\mu_{\rm c}),\notag\\ \alpha_3&=\Big(h+{\rm K}\,\dfrac{h^3}{12}\Big)\,(\mu+\mu_{\rm c}),\qquad \qquad\alpha_4=\Big(h+{\rm K}\,\dfrac{h^3}{12}\Big) (\mu+\mu_{\rm c}),\\
		\beta_1&=2\,\Big(h-{\rm K}\,\dfrac{h^3}{12}\Big)\,\mu\, {L}_{\rm c}^2  \frac{12\,b_3-b_1}{3},\quad 
		\beta_2=\Big(h-{\rm K}\,\dfrac{h^3}{12}\Big)\,\mu\, {L}_{\rm c}^2  (b_1-b_2),\notag\\ \beta_3&=\Big(h-{\rm K}\,\dfrac{h^3}{12}\Big)\,\mu\, {L}_{\rm c}^2  (b_1+b_2), \qquad \  \beta_4=\Big(h-{\rm K}\,\dfrac{h^3}{12}\Big)\,\mu\, {L}_{\rm c}^2  (b_1+b_2).\notag
		\end{align} 
		\item[ii)] We observe that 
		$
		\mu_{\rm c}^{\mathrm{drill}}\,:=\,\alpha_3-\alpha_2\,=\,2\Big(h+{\rm K}\,\dfrac{h^3}{12}\Big)\,\mu_{\rm c}\,,
	$
		which means that the in-plane rotational couple modulus $\,\mu_{\rm c}^{\mathrm{drill}}\,$ of the Cosserat shell model is determined by the Cosserat couple modulus $\,\mu_{\rm c}\,$ of the 3D Cosserat material.

		\item[iii)] In our shell model, the constitutive coefficients are those from the three-dimensional formulation, while the influence of the curved initial shell configuration appears explicitly  in the expression of the coefficients  of the energies for the reduced two-dimensional variational problem.
		
		\item[iv)] The major difference between our model and the previously considered   general 6-parameter shell model is that we include terms up to order $O(h^5)$ and that, even in the case of a simplified theory of order $O(h^3)$, additional mixed terms like the membrane--bending part $\,W_{\mathrm{memb,bend}}\big(  \mathcal{E}_{m,s}^{{\rm lin}} ,\,  \mathcal{K}_{e,s}^{{\rm lin}} \big) \,$ and $W_{\mathrm{curv}}\big(  \mathcal{K}_{e,s} ^{{\rm lin}}  {\rm B}_{y_0} \,  \big) $ are included, which are otherwise difficult to guess.
		
		\item[v)] Beside the fact that mixed membrane--bending terms are included, the constitutive coefficients in our shell model depend on both the  Gau\ss \ curvature ${\rm K}$ and the mean curvature ${\rm H}$, see item i) and compare to \eqref{e90}. Moreover, due to the bilinearity of the density energy, if the final form of the energy density is expressed as a quadratic form  in terms of $\big(  \mathcal{E}_{m,s}^{{\rm lin}} ,\,  \mathcal{K}_{e,s}^{{\rm lin}} \big) $, as in the $W_{\rm EP}\big(  \mathcal{E}_{m,s}^{{\rm lin}},\mathcal{K}_{e,s}^{{\rm lin}}  \big)$, then we remark that the dependence on the mean curvature is not only the effect of the presence of the mixed terms, due to the Cayley-Hamilton equation $
		{\rm B}_{y_0}^2=2\,{\rm H}\, {\rm B}_{y_0}-{\rm K}\, {\rm A}_{y_0}
		$. See for instance the energy term $W_{\mathrm{curv}}\big(  \mathcal{K}_{e,s}^{{\rm lin}}   {\rm B}_{y_0}^2 \,  \big) $ or even  $
		W_{\mathrm{mp}} \big((  \mathcal{E}_{m,s}^{{\rm lin}} \, {\rm B}_{y_0} +  {\rm C}_{y_0} \mathcal{K}_{e,s}^{{\rm lin}} )   {\rm B}_{y_0} \,\big)$ from \eqref{e89l}.
	\end{itemize}
\end{remark}

\bigskip

\begin{footnotesize}
	\noindent{\bf Acknowledgements:}   This research has been funded by the Deutsche Forschungsgemeinschaft (DFG, German Research Foundation) -- Project no. 415894848: NE 902/8-1 (P. Neff) and
	BI 1965/2-1 (M. B\^irsan)). No funding source is specified for the other author.
		
	\medskip

	

	

	\bibliographystyle{plain} 

	\appendix

	\section{The classical linear (first) Koiter membrane-bending model }\setcounter{equation}{0}
	According to \cite[page 344]{Ciarlet00}, \cite[page 154, ]{ciarlet2005introduction}  in the linear (first) Koiter model, the variational problem is to find a midsurface displacement vector field
	$v:\omega\subset\mathbb{R}^2\to\mathbb{R}^3$  minimizing:
	\begin{align}\label{Ap7matrix1}
	\dd\int_\omega &\bigg\{h\bigg(
	\mu\rVert    [\nabla\Theta]^{-T} (\mathcal{G}_{\rm{Koiter}}^{\rm{lin}})^\flat [\nabla\Theta]^{-1}\rVert^2  +\dfrac{\,\lambda\,\mu}{\lambda+2\,\mu} \, \mathrm{tr} \Big[ [\nabla\Theta]^{-T} (\mathcal{G}_{\rm{Koiter}}^{\rm{lin}})^\flat  [\nabla\Theta]^{-1}\Big]^2\bigg) \vspace{6pt} \\
	&+\dd\frac{h^3}{12}\bigg(
	\mu\rVert    [\nabla\Theta]^{-T} \big(\mathcal{R}_{\rm{Koiter}}^{\rm{lin}}\big)^\flat  [\nabla\Theta]^{-1}\rVert^2 +\dfrac{\,\lambda\,\mu}{\lambda+2\,\mu} \, \mathrm{tr} \Big[ [\nabla\Theta]^{-T} \big(\mathcal{R}_{\rm{Koiter}}^{\rm{lin}}\big)^\flat   [\nabla\Theta]^{-1}\Big]^2\bigg)\bigg\}\,{\rm det}(\nabla y_0|n_0)\, {\rm d}a,\notag
	\end{align}
	where $(\mathcal{G}_{\rm{Koiter}}^{\rm{lin}})^\flat $ and $\big(\mathcal{R}_{\rm{Koiter}}^{\rm{lin}}\big)^\flat$ are the lifted quantities of the strain measures \cite{Ciarlet00}   given by
	\begin{equation}
	\label{equ11}
	\mathcal{G}_{\rm{Koiter}}^{\rm{lin}} \,\,:=\frac{1}{2}\big[{\rm I}_m - {\rm I}_{y_0}\big]^{\rm{lin}}= \,\,\frac12\;\big[  (\nabla y_0)^{T}(\nabla v) +  (\nabla v)^T(\nabla y_0)\big]
	= \sym\big[ (\nabla y_0)^{T}(\nabla v)\big]
	\;
	\end{equation}
	and
	\begin{align}
	\label{equ12}
	\mathcal{R}_{\rm{Koiter}}^{\rm{lin}} \,\,:&= \,\, \big[{\rm II}_m - {\rm II}_{y_0}\big]^{\rm{lin}} \,= 	 \Big( \bigl\langle n_0 ,  \partial_{x_\alpha  x_\beta}\,v- \dd\sum_{\gamma=1,2}\Gamma^\gamma_{\alpha \beta}\,\partial_{x_\gamma}\,v\bigr\rangle a^\alpha\,\Big)_{\alpha\beta}\in \mathbb{R}^{2\times 2},
	\end{align}
	The expression of $\mathcal{R}_{\rm{Koiter}}^{\rm{lin}}$ involves the Christoffel  symbols
	$
	\Gamma^\gamma_{\alpha\beta}
	$  on the surface given by
	$
	\Gamma^\gamma_{\alpha\beta}=\bigl\langle a^\gamma, \partial_{x_\alpha} a_\beta\bigr\rangle=-\bigl\langle \partial_{x_\alpha} a^\gamma,  a_\beta\bigr\rangle=\Gamma^\gamma_{\beta\alpha}.
$ Note that,
	using $m=y_0+v$ and
$
	(\nabla m)^T\nabla m=(\nabla y_0)^T\nabla y_0+(\nabla y_0)^T\nabla v+(\nabla v)^T\nabla y_0+\text{h.o.t}\in \mathbb{R}^{2\times 2},
	$
	the linear approximation of the difference  $\frac{1}{2}\big[{\rm I}_m - {\rm I}_{y_0}\big]^{\rm{lin}}$ appearing in the Koiter model  can easily  be obtained \cite[page 92]{Ciarlet00}, the linear approximation of the difference  $\big[{\rm II}_m - {\rm II}_{y_0}\big]^{\rm{lin}}$ needs some more insights from differential geometry \cite[page 95]{Ciarlet00} and it is based on 
	{\it formulas of Gau\ss}  \ $\partial_{x_\alpha} a_\beta=\sum_{\gamma=1,2}\Gamma_{\alpha\beta}^\gamma \,a_\gamma+b_{\alpha\beta}a_3$ and  $  
	\partial_{x_\beta}a^\alpha = - \sum_{\gamma=1,2}\Gamma^\alpha_{\gamma\beta}\,a^\gamma + b^\alpha_\beta\, n_0 
$
	and the formulas of  {\it Weingarten}
$\partial_{x_\alpha}     a_3=\partial_{x_\alpha}   a^3= -\sum_{\beta=1,2}b_{\alpha\beta}\, a^\beta=-\sum_{\beta=1,2} b^\gamma_\beta\, a_\gamma\;
	$
	together with  the relations \cite[page 76]{Ciarlet00}
	$ b_{\alpha\beta}(m)=-\bigl\langle\partial _\alpha a_3(m), a_\beta(m)\bigr\rangle=\bigl\langle\partial _\alpha a_\beta(m), a_3(m)\bigr\rangle=b_{\beta\alpha}(m),
$
	where  $b_{\alpha\beta}(m)$ are the components of the second fundamental form corresponding to the map $m$, $b_{\alpha}^\beta(m)$ are the components of the matrix associated to the Weingarten map (shape operator), and on the following  linear approximation
	$
	n=\,n_0+\frac{1}{\sqrt{\det ((\nabla y_0)^T\nabla y_0)}}\left(\partial_{x_1} y_0\times \partial_{x_2} v+\partial_{x_1} v\times \partial_{x_2} y_0+\text{h.o.t}\right) -\tr(((\nabla y_0)^T\nabla y_0)^{-1}\, \sym ((\nabla y_0)^T\nabla v) )\,n_0.
$
	
	We  note that other alternative forms of the change of metric tensor and  the change of curvature tensor  in \cite[Page 181]{Ciarlet00} are
	\begin{equation}\label{formK}
	\mathcal{G}_{\rm{Koiter}}^{\rm{lin}} =  \Big( \frac{1}{2}(\partial_\beta v_\alpha+\partial_\alpha v_\beta)-\sum_{\gamma=1,2}\Gamma_{\alpha\beta}^\gamma v_\gamma-b_{\alpha\beta}v_3 \Big)_{\alpha\beta}\in \mathbb{R}^{2\times 2},
	\end{equation}
	and \begin{align}\label{formR}
	\mathcal{R}_{\rm{Koiter}}^{\rm{lin}} = & \Big( \partial_{x_\alpha x_\beta}v_3-\sum_{\gamma=1,2}\Gamma_{\alpha\beta}^\gamma \partial_{x_\gamma}v_3-\sum_{\gamma=1,2}b_\alpha^\gamma b_{\gamma\beta}v_3+\sum_{\gamma=1,2}b_{\alpha}^\gamma(\partial_{x_\beta}v_\gamma-\sum_{\tau=1,2}\Gamma_{\beta\gamma}^\tau v_\tau)\\&+\sum_{\gamma=1,2}b_{\beta}^\gamma(\partial_{x_\alpha}v_\gamma-\sum_{\tau=1,2}\Gamma_{\alpha\tau}^\gamma v_\gamma)
	+\sum_{\tau=1,2}(\partial _{x_\alpha}b_\beta^\tau+\sum_{\gamma=1,2}\Gamma_{\alpha\gamma}^\tau b_\beta^\gamma-\sum_{\gamma=1,2}\Gamma_{\alpha\beta}^\gamma b_\gamma^\tau)v_\tau\Big)_{\alpha\beta}\in \mathbb{R}^{2\times 2},\notag
	\end{align}  respectively. Actually, on one hand, the last form of the curvature tensor will be considered when the admissible set of solutions of the variational problem will be defined.  On the other hand, as noticed in \cite[Page 175]{Ciarlet2Diff-Geo2005} by considering the form \eqref{equ12} of the change of metric tensor,   we can impose substantially weaker   regularity assumptions  on the mapping $y_0$. For the linear (first) Koiter model the existence results are given in  \cite[Theorem 7.1.-1 and Theorem 7.1.-2]{Ciarlet00}, see also \cite{blouza1999existence}.

	While the relation between $\mathcal{G}^{\rm{lin}}$ and $ \;\mathcal{G}_{\rm{Koiter}}^{\rm{lin}} $ holds  in the general case, we are able to find  a simple explicit relation between  $\; \mathcal{R}^{\rm{lin}}\; $ and $\; \mathcal{R}_{\rm{Koiter}}^{\rm{lin}} $  only in the case of the constrained Cosserat-shell model. This is not surprising, since only symmetric stress tensors are taken into account in the classical linear Koiter model, i.e., the internal strain energy does not depend on the skew-symmetric part of the considered strain measures (since it is work conjugate to the skew-symmetric part of the stress tensor). In addition,   the linear Koiter model does not consider  extra degrees of freedom.
	In \cite{GhibaNeffPartVI}  we will discuss the choice of the deformation measures in  shell models. We  will see that the classical strain measure  $\; \mathcal{R}_{\rm{Koiter}}^{\rm{lin}} $ (the classical bending strain measure, also known as the change of curvature tensor) does not represent the unique choice and that some other modified expressions of the classical bending tensor   may be more suitable in the modelling of a shell.
	
	\end{footnotesize}

\end{document}